\documentclass[10pt,reqno,a4paper]{article} 

\usepackage{anysize}
\marginsize{3.0cm}{3.0cm}{2.2cm}{2.2cm}

\usepackage[english]{babel}
\usepackage[centertags]{amsmath}
\usepackage{amsfonts,amssymb,amsthm}
\usepackage{color}
\usepackage{hyperref} 
\usepackage{mathrsfs}
\usepackage[sans]{dsfont}

\let\OLDthebibliography\thebibliography
\renewcommand\thebibliography[1]{
  \OLDthebibliography{#1}
  \setlength{\parskip}{0pt}
  \setlength{\itemsep}{0pt plus 0.3ex} }

\numberwithin{equation}{section}

\theoremstyle{plain}
\newtheorem{theorem}{Theorem}[section]

\newtheorem{lemma}[theorem]{Lemma}

\theoremstyle{definition}
\newtheorem{remarkbase}[theorem]{Remark}
\newenvironment{remark}{\pushQED{\qed} \remarkbase}{\popQED\endremarkbase}

\newcommand{\R}{{\mathbb R}}

\newcommand{\Z}{\mathbb Z}
\newcommand{\Q}{\mathbb Q}
\newcommand{\T}{{\mathbb T}}

\newcommand{\mA}{\mathcal{A}}
\newcommand{\mB}{\mathcal{B}}
\newcommand{\mC}{\mathcal{C}}
\newcommand{\mD}{\mathcal{D}}

\newcommand{\mF}{\mathcal{F}}

\newcommand{\mH}{\mathcal{H}}

\newcommand{\mK}{\mathcal{K}}

\newcommand{\mN}{\mathcal{N}}
\newcommand{\mP}{\mathcal{P}}

\newcommand{\mR}{\mathcal{R}}
\newcommand{\mS}{\mathcal{S}}
\newcommand{\mT}{\mathcal{T}}
\newcommand{\mU}{\mathcal{U}}

\renewcommand{\a}{\alpha}
\renewcommand{\b}{\beta}
\newcommand{\g}{\gamma}
\renewcommand{\d}{\delta}

\newcommand{\e}{\varepsilon}
\newcommand{\ph}{\varphi}
\newcommand{\lm}{\lambda}

\newcommand{\Om}{\Omega}
\newcommand{\om}{\omega}

\newcommand{\s}{\sigma}

\renewcommand{\th}{\vartheta}

\newcommand{\la}{\langle}
\newcommand{\ra}{\rangle}
\newcommand{\pa}{\partial}
\renewcommand{\div}{\mathrm{div}\,}
\newcommand{\grad}{\nabla}

\newcommand{\bcb}{\begin{color}{blue}}
\newcommand{\bcr}{\begin{color}{red}}
\newcommand{\bcg}{\begin{color}{green}}
\newcommand{\ec}{\end{color}}

\title{Nearly toroidal, periodic and quasi-periodic motions \\ 
of fluid particles driven by the Gavrilov solutions \\ 
of the Euler equations}
 
\author{\normalsize{Pietro Baldi}} 
\date{} 
\begin{document}

\maketitle

\textbf{Abstract.} 
We consider the smooth, compactly supported solutions 
of the steady 3D Euler equations of incompressible fluids contructed by Gavrilov in 2019, 
and we study the corresponding fluid particle dynamics. 
This is an \textsc{ode} analysis, which contributes to the description of Gavrilov's vector field.

\begin{small}
\tableofcontents
\end{small}

\section{Introduction and main result}

In the remarkable paper \cite{Gav}, Gavrilov proved the existence
of a nontrivial solution of $C^\infty$ class, 
with compact support, of the steady Euler equations 
of incompressible fluids in $\R^3$. 
The result in \cite{Gav} is important and surprising, 
because previously, 
on the basis of some negative partial results, 
it was conjectured that compactly supported, nontrivial, smooth solutions 
of the 3D steady Euler equations cannot exist:
see the clear explanation at the beginning of \cite{Constantin.ecc}
and the general discussion about existence 
of compactly supported smooth solutions in \cite{Peralta-Lioville}.
In addition, another reason of interest for the fruitful construction of \cite{Gav} is that 
recently it has been used as a building block to produce other interesting solutions, 
both stationary and time-dependent, of the Euler equations of fluid dynamics,  
see Section \ref{subsec:lit} below. 

\medskip

Now suppose that a fluid moves according 
to the Gavrilov solution of the Euler equations, 
that is, suppose that the fluid particles are driven by Gavrilov's velocity vector field.
Which movement of the fluid do we observe? 
Of course the particles outside the support of the vector field do not move at all, 
but how do they move in the region where the field is nonzero? 

\medskip

In the present paper we deal with this question. 
It turns out that every fluid particle travels along a trajectory  
that lies on a nearly toroidal surface, which is a level set of the pressure. 
The motion of every fluid particle is periodic or quasi-periodic in time; 
we prove that there are both periodic and quasi-periodic motions, 
and the value of the pressure 
determines whether the trajectories on its level set  
are all periodic or all quasi-periodic. 

In fact, the system of differential equations in $\R^3$ 
describing the motion of the fluid particles turns out to be 
integrable, as it can be transformed into the system of 
a Hamiltonian system of one degree of freedom 
and a third equation that can be directly solved by integration.
We write the Hamiltonian system in angle-action coordinates $(\s, I)$, 
and prove that there exists a change of variables 
on a neighborhood of the support of Gavrilov's vector field 
such that the equations of motion in the new coordinates $(\s, \b, I)$ 
becomes 
\[
\dot \s = \Om_1(I), \quad \ 
\dot \b = \Om_2(I), \quad \ 
\dot I = 0,
\]
where $\s$ and $\b$ are angle variables rotating with constant angular velocities 
$\Om_1(I)$, $\Om_2(I)$, and $I$ is a constant action variable, 
which is, in fact, a reparametrization of the pressure. 
The full statement is in Theorem \ref{thm:main}.

\subsection{The Gavrilov solutions of the steady Euler equations}
\label{subsec:Gav.sol}

In the main part of the construction in \cite{Gav}, 
given any $R > 0$, the circle 
\begin{equation*} 
\mC := \{ (x,y,z) \in \R^3 : x^2 + y^2 = R^2, \ z = 0 \}
\end{equation*}
in $\R^3$ is considered, and, in an open neighborhood $\mN$ of $\mC$, 
two functions $U,P$ are defined, 
$U : \mN \to \R^3$ and $P : \mN \to \R$, 
both in $C^\infty(\mN \setminus \mC)$,
solving the steady Euler equations 
\begin{equation} \label{Euler}
U \cdot \grad U + \grad P = 0, 
\quad \ 
\div U = 0
\end{equation} 
in $\mN \setminus \mC$, 
together with the fundamental ``localizability condition'' 
\begin{equation} \label{loc} 
U \cdot \grad P = 0.
\end{equation} 
As a final step of the proof,  
the functions $U,P$ are multiplied by smooth cut-off functions 
to obtain $C^\infty(\R^3)$ functions $\tilde U, \tilde P$, 
where $\tilde U$ and $\grad \tilde P$ have compact support 
contained in $\mN \setminus \mC$, solving \eqref{Euler} (and also \eqref{loc}) in $\R^3$. 

Let us be more precise. Denote $\rho := \sqrt{x^2 + y^2}$.  
For $\d \in (0,R)$, let 
\begin{equation} \label{def.mN}
\mN = \{ (x,y,z) \in \R^3 : (\rho - R)^2 + z^2 < \delta^2 \}.
\end{equation}
In $\mN$, the solution $(U,P)$ of \cite{Gav} is given by
\begin{equation}  \label{def.U.P}
U(x,y,z) = 
u_\rho(\rho, z) e_\rho(x,y) 
+ u_\ph(\rho, z) e_\ph(x,y) 
+ u_z(\rho, z) e_z,
\quad P(x,y,z) = p(\rho,z),
\end{equation}
where 
\begin{alignat}{3} 
e_\rho(x,y) & = \frac{1}{\rho} (x, y, 0), & \quad \ 
e_\ph(x,y) & = \frac{1}{\rho} (-y, x, 0), & \quad \ 
e_z & = (0,0,1),
\notag \\ 
u_\rho(\rho,z) & = \frac{ \pa_z p(\rho,z) }{ \rho }, & \quad \ 
u_\ph(\rho,z) & = \frac{b(\rho,z)}{\rho}, & \quad \ 
u_z(\rho,z) & = - \frac{ \pa_\rho p(\rho,z) }{ \rho },  
\notag \\  
b(\rho,z) & = \frac{R^3}{4} \sqrt{H(a(\rho,z))}, & \quad \ 
p(\rho,z) & = \frac{R^4}{4} a(\rho,z), & \quad \ 
a(\rho,z) & = \a \Big( \frac{\rho}{R}, \frac{z}{R} \Big),
\label{Gav.a.b.p}
\end{alignat}
and $\a,H$ are functions defined in \cite{Gav} 
in terms of solutions of certain differential equations; 
$H(0) = 0$, and $H$ is analytic in a neighborhood of $0$;
$\a$ has a strict local minimum at $(1,0)$, 
with $\a(1,0) = 0$, 
and it is analytic in a neighborhood of $(1,0)$. 
Hence $\a$ and $H \circ \a$ are both well-defined and analytic 
in a disc of $\R^2$ of center $(1,0)$ and radius $r_0$, 
for some universal constant $r_0 > 0$ 
(where ``universal'' means that $r_0$ does not depend on anything). 
If $\d$ in \eqref{def.mN} satisfies 
\begin{equation*} 
\delta \leq r_0 R,
\end{equation*}
then $a(\rho,z)$ and $H(a(\rho,z))$, 
where $\rho = \sqrt{x^2 + y^2}$, 
are well-defined and analytic in $\mN$
(note that, in $\mN$, one has 
$0 < R-\delta < \rho < R+\delta$; 
in particular, $\rho$ is bounded away from zero). 
Also, $b(\rho,z)$ is well-defined and continuous in $\mN$, 
and (because of the square root $\sqrt{H}$) 
it is analytic in $\mN \setminus \mC$. 
Hence $P$ is analytic in $\mN$, 
while $U$ is continuous in $\mN$ 
and analytic in $\mN \setminus \mC$.

The solution $(\tilde U, \tilde P)$ in \cite{Gav} 
is defined in $\mN$ as 
\begin{equation}  \label{def.tilde.U.tilde.P}
\tilde U(x,y,z) = \om(P(x,y,z)) U(x,y,z), 
\quad \ 
\tilde P(x,y,z) = W(P(x,y,z)),
\end{equation}
where $\om : \R \to \R$ is any $C^\infty$ function vanishing outside 
the interval $[\e, 2 \e]$, with $\e > 0$ small enough, 
and $W : \R \to \R$ is a primitive of $\om^2$; for example, 
\begin{equation} \label{def.W}
W(s) = \int_0^s \om^2(\s) \, d\s.
\end{equation} 
Then $(\tilde U, \tilde P)$ is extended to $\R^3$ 
by defining 
$(\tilde U, \tilde P) = (0, W(2\e))$ in $\R^3 \setminus \mN$. 

Note that $\tilde U$ and $\grad \tilde P$ can be nonzero only in the set 
\begin{equation}  \label{def.mS}
\mS := \{ (x,y,z) \in \mN : \e < P(x,y,z) < 2 \e \},
\end{equation}
and, if $\om(s)$ is nonzero at some $s \in (\e, 2\e)$, 
then both $\tilde U$ and $\grad \tilde P$ are actually nonzero 
at the corresponding points in $\mS$. 
Moreover, $P = 0$ on the circle $\mC$, 
and, for $\e$ small enough, 
$P > 3 \e$ at all points of $\mN$ sufficiently far from $\mC$;
more precisely, to fix the details of the construction, 
we introduce a parameter $\tau > 0$ and assume that 
\begin{equation} \label{P.larger}
P(x,y,z) > \tau 
\quad \ \forall (x,y,z) \in \mN \setminus \mN'
\end{equation}
and $\tau \geq 3 \e$, where
\begin{equation*} 
\mN' := \{ (x,y,z) \in \R^3 : (\rho-R)^2 + z^2 < (\d/4)^2 \}.
\end{equation*}
Thus, the closure of $\mS$ is contained in the open set 
\begin{equation}  \label{def.mS.star}
\mS^* := \{ (x,y,z) \in \mN : 0 < P(x,y,z) < \tau \},
\end{equation}
and $\mS^* \subseteq (\mN' \setminus \mC)
\subseteq (\mN \setminus \mC)$.

\subsection{Main result: description of the fluid particle dynamics} 

A preliminary, basic observation regarding the solutions of Section \ref{subsec:Gav.sol}
is that any such solution with a given $R>0$ can be obtained, 
by rescaling, from another one having $R=1$. 
Even more, we show that Gavrilov's solutions are a 1-parameter subset 
of a 2-parameter family of solutions, where $R$ plays a dual role 
related to two different scaling invariances of the Euler equations, 
both preserving the localizability condition \eqref{loc}. 
These basic observations are in Section \ref{sec:dual role} 
(see Lemma \ref{lemma:fam}). 
Thanks to these properties, we study the motion of the fluid particles 
in the normalized case $R=1$; the motion for any other $R>0$ 
is immediately obtained by rescaling the amplitude and the time variable, 
as explained in Lemma \ref{lemma:ode.R}.

\medskip

To study the motion of the fluid particles driven 
by the Gavrilov vector field $\tilde U$ 
means to study the solutions $\R \to \R^3$, $t \mapsto (x(t), y(t), z(t))$ 
of the system
\begin{equation}  \label{ode.xyz}
(\dot x(t), \dot y(t), \dot z(t)) = \tilde U(x(t), y(t), z(t)),
\end{equation}
which is an autonomous \textsc{ode} in $\R^3$. 
The dot above a function denotes its time derivative. 
Before dealing with system \eqref{ode.xyz}, 
we recall some definitions about quasi-periodic functions. 

\medskip

A vector $\Om = (\Om_1, \ldots, \Om_n) \in \R^n$, $n \geq 1$, 
is said to be \emph{rationally independent} if 
$\Om \cdot k = \Om_1 k_1 + \ldots + \Om_n k_n$ 
is nonzero for all integer vectors $k = (k_1, \ldots, k_n) \in \Z^n \setminus \{ 0 \}$. 

Given a set $X$, a function $v : \R \to X$, $t \mapsto v(t)$ is said to be 
\emph{quasi-periodic with frequency vector} 
$\Om = (\Om_1, \ldots, \Om_n) \in \R^n$, $n \geq 2$, 
if $\Om$ is rationally independent  
and there exists a function 
$w : \R^n \to X$, $2\pi$-periodic in each real variable, 
such that $v(t) = w(\Om_1 t, \ldots, \Om_n t)$ for all $t \in \R$.
Moreover, to ensure that the number $n$ is not higher than necessary, 
we add the condition that there does not exist any vector 
$\tilde \Om = (\tilde \Om_1, \ldots, \tilde \Om_m) \in \R^m$, 
with $m < n$, 
and any function $\tilde w : \R^m \to X$, 
$2\pi$-periodic in each real variable, 
such that $v(t) = \tilde w (\tilde \Om_1 t, \ldots, \tilde \Om_m t)$ 
for all $t \in \R$. 

For example, for $X=\R$, $n=3$, and $\Om = (\Om_1, \Om_2, \Om_3) \in \R^3$ 
a rationally independent vector, 
if $w( \th_1, \th_2, \th_3) = \cos(\th_1 + \th_2) + \cos(\th_3)$, 
then $n=3$ is not minimal, as it can be reduced to $n=2$ by taking 
$\tilde w(\th_1, \th_2) = \cos(\th_1) + \cos(\th_2)$, 
$\tilde \Om = (\tilde \Om_1, \tilde \Om_2) \in \R^2$, 
with $\tilde \Om_1 = \Om_1 + \Om_2$ and $\tilde \Om_2 = \Om_3$, 
while $n=2$ cannot be further reduced. 
Hence the function $v(t) = \tilde w(\tilde \Om_1 t, \tilde \Om_2 t) 
= \cos(\tilde \Om_1 t) + \cos (\tilde \Om_2 t)$ 
is quasi-periodic with frequency vector $\tilde \Om \in \R^2$. 

For $n=2$, a vector $\Om = (\Om_1, \Om_2) \in \R^2$ is rationally independent 
if and only if $\Om_1$ is nonzero and the ratio $\Om_2 / \Om_1$ is irrational. 
Hence a function $v(t) = w(\Om_1 t, \Om_2 t)$ is quasi-periodic 
with frequency vector $\Om = (\Om_1, \Om_2)$ 
if $w(\th_1, \th_2)$ is $2\pi$-periodic in $\th_1$ and in $\th_2$, 
$\Om_1$ is nonzero,
$\Om_2 / \Om_1$ is irrational, 
and $v(t)$ is not a periodic function.

\medskip

The main result of this paper is the following description 
of Gavrilov's fluid particle dynamics.

\begin{theorem} \label{thm:main}
There exist universal positive constants $\d_0, \tau_0, \e_0, I^*$ with the following properties.
Let 
\[
\mC, \d, \mN, U, P, \e, \om, W, \tilde U, \tilde P, \mS, \tau, \mN', \mS^*
\] 
be the sets, constants, and functions 
defined in Section \ref{subsec:Gav.sol} for $R=1$, with 
$\d = \d_0$, $\tau = \tau_0$, and $0 < \e \leq \e_0$.  

\medskip

$(i)$ There exists an analytic diffeomorphism 
\begin{equation*} 
\Phi : \T \times \T \times (0, I^*) \to \mS^*, 
\end{equation*}
where $\T := \R / 2 \pi \Z$, such that 
the change of variable $(x,y,z) = \Phi(\s,\b,I)$ 
transforms system \eqref{ode.xyz} into a system of the form
\begin{equation}  \label{trans.syst.nel.teorema}
\dot \s = \Om_1(I), \quad \ 
\dot \b = \Om_2(I), \quad \ 
\dot I = 0.
\end{equation}
As a consequence, the solution $(x(t), y(t), z(t))$ 
of the Cauchy problem \eqref{ode.xyz} with initial datum 
\begin{equation}  \label{init.data.nel.teorema}
(x(0), y(0), z(0)) = (x_0, y_0, z_0) = \Phi(\s_0, \b_0, I_0) \in \mS^*
\end{equation}
is the function 
\begin{equation}  \label{sol.xyz.nel.teorema}
(x(t), y(t), z(t)) = \Phi( \s(t), \b(t), I(t)),  
\end{equation}
defined for all $t \in \R$, where 
\begin{equation}  \label{sol.good.nel.teorema}
\s(t) = \s_0 + \Om_1(I_0) t, \quad \ 
\b(t) = \b_0 + \Om_2(I_0) t, \quad \ 
I(t) = I_0.
\end{equation}
The angle variables $\s(t), \b(t) \in \T$ 
rotate with constant angular frequencies $\Om_1(I_0), \Om_2(I_0)$ respectively, 
and the variable $I(t) = I_0$ is constant in time. 

\medskip

$(ii)$ 
The first and third equations of the transformed system \eqref{trans.syst.nel.teorema}
form a Hamiltonian system 
\begin{equation}  \label{Ham.syst.nel.teorema}
\dot \s = \pa_I \mH(\s, I), \quad \dot I = - \pa_\s \mH(\s, I)
\end{equation}
with Hamiltonian $\mH(\s, I) = \mH(I)$ independent of the angle variable $\s$; 
hence \eqref{Ham.syst.nel.teorema} is 
a Hamiltonian system in angle-action variables. 

\medskip

$(iii)$
The frequency $\Om_1(I)$ is given by 
\begin{equation*} 
\Om_1(I) = \om ( \mK(I) ) \, \mK'(I)
\end{equation*}
where $\mK(I)$ is the restriction to the interval $(0, I^*)$ of
a function defined and analytic in the interval $(-I^*, I^*)$, with Taylor expansion 
\[
\mK(I) = I + \frac{1065}{1024} I^3 + O(I^4)
\]
around $I=0$, and strictly increasing in $(-I^*, I^*)$. 
The frequency $\Om_2(I)$ is given by 
\begin{equation*} 
\Om_2(I) = \sqrt{I} \, \mR(I) \Om_1(I) 
\end{equation*}
where $\mR(I)$ is the restriction to the interval $(0, I^*)$ of
a function defined and analytic in the interval $(-I^*, I^*)$, 
with Taylor expansion 
\[
\mR(I) = 1 + \frac74 I + O(I^2)
\]
around $I=0$. 
If $\om( \mK(I)) \neq 0$, 
then both $\Om_1(I)$ and $\Om_2(I)$ are nonzero, with ratio 
\begin{equation*} 
\frac{\Om_2(I)}{\Om_1(I)} = \sqrt{I} \, \mR(I).
\end{equation*}

The function $I \mapsto \sqrt{I} \, \mR(I)$ is strictly increasing and analytic 
in $(0, I^*)$. Therefore it is 
rational for infinitely many values of $I$, 
and irrational for infinitely many other values of $I$. 
More precisely, 
the set $\{ I \in (0, I^*) : \sqrt{I} \, \mR(I) \in \Q \}$ is a countable set, 
while the set $\{ I \in (0,I^*) : \sqrt{I} \, \mR(I) \notin \Q \}$ 
has full Lebesgue measure.

\medskip

$(iv)$ 
For $\om( \mK(I_0) ) \neq 0$, 
the solution \eqref{sol.xyz.nel.teorema} 
of the Cauchy problem \eqref{ode.xyz}, \eqref{init.data.nel.teorema}
is periodic in time if $\sqrt{I_0}\, \mR(I_0)$ is rational, 
and it is quasi-periodic in time 
with frequency vector $(\Om_1(I_0), \Om_2(I_0))$ 
if $\sqrt{I_0}\, \mR(I_0)$ is irrational.

\medskip

$(v)$
The map $\Phi(\s,\b,I)$ admits a converging expansion in powers of $\sqrt{I}$ around $I=0$;
more precisely, there exists a map $\Psi(\s, \b, \mu)$, defined and analytic 
in $\T^2 \times (- \mu_0, \mu_0)$, where $\mu_0 = \sqrt{I^*}$, such that 
$\Phi(\s, \b, I) = \Psi(\s, \b, \sqrt{I})$ 
for all $(\s, \b, I) \in \T^2 \times (0, I^*)$.
The map $\Phi(\s, \b, I)$ has the form 
\[
\Phi(\s, \b, I) = 
\begin{pmatrix} 
\rho(\s, I) \cos(\b + \eta(\s,I)) \\ 
\rho(\s, I) \sin(\b + \eta(\s,I)) \\ 
\zeta(\s, I) 
\end{pmatrix}
\]
where the functions $\rho(\s,I), \eta(\s,I), \zeta(\s,I)$ have expansion
\[
\rho(\s,I) = 1 + \sqrt{2I} \sin \s + O(I), \quad \ 
\eta(\s, I) = O(I), \quad \ 
\zeta(\s, I) = \sqrt{2I} \cos \s + O(I).
\]

\medskip

$(vi)$ 
The action variable $I$ and the pressure $P$ in \eqref{def.U.P} are related by the identity 
\[
P(\Phi(\s, \b, I)) = \mK(I) 
\quad \ \forall (\s, \b, I) \in \T^2 \times (0, I^*).
\]
The pressure level and the action are in bijective correspondence; 
thus, the action is a reparametrization 
of the pressure. 
The frequencies $\Om_1(I), \Om_2(I)$ 
could also be expressed in terms of the pressure level. 
The pressure $\tilde P$ in \eqref{def.tilde.U.tilde.P} 
satisfies $\tilde P(\Phi(\s,\b,I)) = W( \mK(I) )$. 
 
\medskip

$(vii)$
The trajectory $\{ (x(t), y(t), z(t)) : t \in \R \}$ 
of the solution \eqref{sol.xyz.nel.teorema} lies in the level set 
\[
\mT_{\ell} = \{ (x,y,z) \in \R^3 : P(x,y,z) = \ell \}
\]
of the pressure, where the value $\ell = P(x_0, y_0, z_0)$ 
is determined by the initial datum \eqref{init.data.nel.teorema}. 
The level set $\mT_\ell$ has a nearly toroidal shape, 
because $P(x,y,z)$ is given by \eqref{def.U.P}, \eqref{Gav.a.b.p}, 
and 
\[
\a(\rho,z) = 2(\rho-1)^2 + 2 z^2 + O( ( |\rho-1| + |z| )^3 )
\]
around $(\rho,z) = (1,0)$. 
More precisely, $\mT_\ell$ is the diffeomorphic image 
\[
\mT_\ell = \Phi( \T^2 \times \{ I \} ) 
\]
of the 2-dimensional torus 
$\T^2 \times \{ I \} = \{ (\s, \b, I) : (\s , \b) \in \T^2 \}$,
where $I$ is determined by the identity 
\begin{equation} \label{ell.I.relation}
\ell = \mK(I). 
\end{equation}

The pressure level $\ell$ 
determines whether the solution \eqref{sol.xyz.nel.teorema}, \eqref{sol.good.nel.teorema}
of the Cauchy problem \eqref{ode.xyz} 
with initial datum \eqref{init.data.nel.teorema} 
on the surface $\mT_\ell$ is periodic or quasi-periodic, 
depending on the rationality/irrationality of the ratio 
$\Om_2(I) / \Om_1(I)$, 
where $I$ and $\ell$ are related by \eqref{ell.I.relation}. 
Different solutions of system \eqref{ode.xyz} 
lying on the same surface $\mT_\ell$, i.e., having the same pressure level, 
share the same frequencies $\Om_1(I), \Om_2(I)$ 
and, therefore, the same frequency ratio $\Om_2(I) / \Om_1(I)$.    
\end{theorem}

\begin{remark}
Theorem \ref{thm:main} is stated for cut-off functions $\om$ 
supported in $[\e, 2\e]$, like those in \cite{Gav}; 
however, Theorem \ref{thm:main}, as well as the result of \cite{Gav},
also holds for $\om$ supported in any interval $[\e_1, \e_2]$ 
with $0 < \e_1 < \e_2 \leq \e_0$, without changing anything in the proof. 
\end{remark}

\begin{remark}
By \eqref{Ham.syst.nel.teorema}, 
system \eqref{trans.syst.nel.teorema} is equivalent to the integrable Hamiltonian system 
with two degrees of freedom in angle-action variables 
\[
\dot \s = \pa_I \mH^+, \quad 
\dot I = - \pa_\s \mH^+, \quad 
\dot \b = \pa_K \mH^+, \quad 
\dot K = - \pa_\b \mH^+
\]
with Hamiltonian $\mH^+(\s, I, \b, K) = \mH^+(I,K) = \mH(I) + \Om_2(I) K$,
restricted to the invariant set $K = 0$. 
\end{remark}

Theorem \ref{thm:main} is proved in Sections \ref{sec:conj} and \ref{sec:Taylor},
splitting the proof into several short simple steps. 
The proof uses only basic tools from the classical theory of 
\textsc{ode}s and dynamical systems.

\subsection{Related literature}
\label{subsec:lit}

A general discussion about the existence 
of compactly supported smooth solutions of \textsc{pde}s, 
also in comparison with the result of Gavrilov for the steady Euler equations, 
can be found in the recent preprint \cite{Peralta-Lioville};
in particular, for Navier-Stokes equations, see \cite{Koch.etc}.

The nice, explicit construction of Gavrilov's original paper \cite{Gav}
has been revisited and generalized by Constantin, La and Vicol 
in \cite{Constantin.ecc}. 
The paper \cite{Constantin.ecc} also uses the result by Gavrilov 
as a building block to prove the existence of compactly supported solutions 
of the steady Euler equations 
that have a given H\"older regularity $C^\a$ 
and are not in $C^\b$ for any $\b > \a$.  
The proof employs the invariances of the Euler equations 
and the fact that the sum of compactly supported solutions 
with disjoint supports is itself a solution.

The result by Gavrilov has also been used recently  
by Enciso, Peralta-Salas and Torres de Lizaur in \cite{Peralta.Gav}
to produce time-quasi-periodic solutions of the 3D Euler equations. 
In \cite{Peralta.Gav} the authors extend to the 3D case, 
and to the $n$-dimensional case for all $n$ even,
the construction by Crouseilles and Faou \cite{Faou.ecc} 
of time-quasi-periodic solutions of the 2D Euler equations.
Both \cite{Faou.ecc} and \cite{Peralta.Gav} 
use in a clever way 
the compactly supported solutions of the steady equations 
as the main ingredients of the construction. 

The time-quasi-periodic solutions in \cite{Faou.ecc} and \cite{Peralta.Gav} 
are not of \textsc{kam} type, that is, they are constructed outside the context 
of the Kolmogorov-Arnold-Moser perturbation theory of nearly integrable dynamical systems, 
where small divisor problems typically appear. 
Time-quasi-periodic solutions of the 3D Euler equations 
of \textsc{kam} type are obtained in \cite{Baldi.Montalto} by Montalto and the author 
in presence of a forcing term, using pseudo-differential calculus 
and techniques of \textsc{kam} theory for \textsc{pde}s. 

By Theorem \ref{thm:main}, the domain $\mS^*$ 
is foliated by the 2-dimensional tori $\mT_\ell$, 
invariant for the vector field $\tilde U$, 
on which the motion is periodic or quasi-periodic.  
This is true not only for the Gavrilov vector field, 
but for all steady 3D Euler flows under suitable assumptions, 
by a theorem of Arnold \cite{Arnold.1965}, \cite{Arnold.1966}. 
Arnold's theorem, and its role in the study of the mixing property for Euler flows, 
is discussed by Khesin, Kuksin and Peralta-Salas in \cite{Kuk.K.Peralta}.

\section{Dual role of the parameter $R$ in Gavrilov's solutions}
\label{sec:dual role}

Two different roles are played simultaneously by the parameter $R$ 
in the Gavrilov's solutions, because $R$ is 
\begin{itemize}
\item both a rescaling factor for the independent variable $(x,y,z) \in \R^3$, 
appearing as $R^{-1}$ in the argument of $\a$ in \eqref{Gav.a.b.p},

\item and an amplitude factor multiplying the vector fields $U, \tilde U$ 
and the pressures $P, \tilde P$, 
appearing as powers $R^3$ and $R^4$ in the definition 
of $b$ and $p$ in \eqref{Gav.a.b.p}.
\end{itemize}
Separating these two different scalings helps to clarify the role 
of the parameters in Gavrilov's construction. 

In fact, we observe that there exists a family of solutions 
described by two real parameters $(\lm, \mu)$
such that the solutions of Section \ref{subsec:Gav.sol} 
are obtained in the special case $(\lm, \mu) = (R, R^2)$.
This means that, regarding the parameter $R$, 
Gavrilov's solutions form a 1-parameter subset of a 2-parameter family of solutions.
Each element of the family is obtained from any other element of the family
by two simple rescalings, which correspond to two basic invariances 
of the Euler equations, also preserving the localizability condition. 
This allows us to study only one element of the family
(in particular, a normalized one), obtaining directly a description of all the other elements.

Given $R>0$, let 
\begin{equation}  \label{def.mU}
\mU := (\mC, \d, \mN, U, P, \e, \om, W, \tilde U, \tilde P, \mS, \tau, \mN', \mS^*)
\end{equation}
denote the list of the elements (sets, constants, and functions) 
defined in Section \ref{subsec:Gav.sol}. 
For every $\lm > 0$ and $\mu > 0$, 
we define a rescaled version of each element of the list $\mU$ in the following way.   
We define 
\begin{align}
\mA_{\lm,\mu} \mC 
& := \{ (x,y,z) \in \R^3 : \rho = \lm R, \ z = 0 \},
\notag \\
\mA_{\lm,\mu} \d 
& := \lm \d, 
\notag \\
\mA_{\lm,\mu} \mN
& := \{ (x,y,z) \in \R^3 : (\rho - \lm R)^2 + z^2 < (\lm \d)^2 \},
\notag \\
(\mA_{\lm,\mu} U)(x,y,z) 
& := \mu U \Big( \frac{x}{\lm}, \frac{y}{\lm}, \frac{z}{\lm} \Big) 
\quad \ \forall (x,y,z) \in \mA_{\lm,\mu} \mN,
\notag \\
(\mA_{\lm,\mu} P)(x,y,z) 
& := \mu^2 P \Big( \frac{x}{\lm}, \frac{y}{\lm}, \frac{z}{\lm} \Big)
\quad \ \forall (x,y,z) \in \mA_{\lm,\mu} \mN,
\notag \\
\mA_{\lm,\mu} \e
& := \mu^2 \e, 
\notag \\
(\mA_{\lm,\mu} \om)(s)
& := \om \Big( \frac{s}{\mu^2} \Big)
\quad \ \forall s \in \R,
\notag \\
(\mA_{\lm,\mu} W)(s)
& := \mu^2 W \Big( \frac{s}{\mu^2} \Big) 
\quad \ \forall s \in \R,
\notag \\
(\mA_{\lm,\mu} \tilde U)(x,y,z) 
& := \mu \tilde U \Big( \frac{x}{\lm}, \frac{y}{\lm}, \frac{z}{\lm} \Big) 
\quad \ \forall (x,y,z) \in \R^3, 
\notag \\
(\mA_{\lm,\mu} \tilde P)(x,y,z) 
& := \mu^2 \tilde P \Big( \frac{x}{\lm}, \frac{y}{\lm}, \frac{z}{\lm} \Big),
\quad \ \forall (x,y,z) \in \R^3, 
\notag \\
\mA_{\lm,\mu} \mS 
& := \{ (x,y,z) \in \mA_{\lm,\mu} \mN : 
\mA_{\lm,\mu} \e < (\mA_{\lm,\mu} P)(x,y,z) < 2 \mA_{\lm,\mu} \e \},
\notag \\
\mA_{\lm,\mu} \tau 
& := \mu^2 \tau,
\notag \\ 
\mA_{\lm,\mu} \mN'
& := \{ (x,y,z) \in \R^3 : (\rho - \lm R)^2 + z^2 < (\lm \d / 4)^2 \},
\notag \\
\mA_{\lm,\mu} \mS^*
& := \{ (x,y,z) \in \mA_{\lm,\mu} \mN : 
0 < (\mA_{\lm,\mu} P)(x,y,z) < 3 \mA_{\lm,\mu} \e \},
\label{def.rescaling.op}
\end{align}
where $\rho = \sqrt{x^2 + y^2}$. 
We denote by $\mA_{\lm,\mu} \mU := (\mA_{\lm,\mu} \mC, \ldots, \mA_{\lm,\mu} \mS^*)$ 
the list of the rescaled elements. 
The properties of $\mU$ in Section \ref{subsec:Gav.sol} 
become the following properties for $\mA_{\lm,\mu} \mU$. 

\begin{itemize}
\item
The constant $\mA_{\lm,\mu} \d = \lm \d$ satisfies 
$0 < \lm \d < \lm R$ and $\lm \d \leq \lm r_0 R$.

\item
The rescaled pressure $\mA_{\lm,\mu} P$ is analytic in $\mA_{\lm,\mu} \mN$.

\item
The rescaled vector field $\mA_{\lm,\mu} U$ is continuous in $\mA_{\lm,\mu} \mN$ 
and analytic in $(\mA_{\lm,\mu} \mN) \setminus (\mA_{\lm,\mu} \mC)$.

\item
The pair $(\mA_{\lm,\mu} U, \, \mA_{\lm,\mu} P)$ 
satisfies the Euler equations and the localizability condition in 
$(\mA_{\lm,\mu} \mN) \setminus (\mA_{\lm,\mu} \mC)$.

\item
The function $\mA_{\lm,\mu} \om$ is $C^\infty(\R,\R)$ with support 
contained in $[\mA_{\lm,\mu} \e, 2 \mA_{\lm,\mu} \e] = [\mu^2 \e, 2 \mu^2 \e]$.

\item
The function $\mA_{\lm,\mu} W$ satisfies 
\[
(\mA_{\lm,\mu} W)(s) = \int_0^s (\mA_{\lm,\mu} \om)^2(\s) \, d\s 
\quad \ \forall s \in \R.
\]

\item
The vector field $\mA_{\lm, \mu} \tilde U$ satisfies 
\begin{align*}
(\mA_{\lm,\mu} \tilde U)(x,y,z) 
& = \mu \om \Big( P \Big( \frac{x}{\lm}, \frac{y}{\lm}, \frac{z}{\lm} \Big) \Big)
U \Big( \frac{x}{\lm}, \frac{y}{\lm}, \frac{z}{\lm} \Big)
\\
& = (\mA_{\lm,\mu} \om) \big( (\mA_{\lm,\mu} P)(x,y,z) \big) \cdot (\mA_{\lm,\mu} U)(x,y,z) 
\quad \ \forall (x,y,z) \in \mA_{\lm,\mu} \mN
\end{align*}
and $(\mA_{\lm,\mu} \tilde U)(x,y,z) = 0$ 
for all $(x,y,z) \in \R^3 \setminus (\mA_{\lm,\mu} \mN)$.

\item
The pressure $\mA_{\lm,\mu} \tilde P$ satisfies 
\begin{align*}
(\mA_{\lm,\mu} \tilde P)(x,y,z) 
& = \mu^2 W \Big( P \Big( \frac{x}{\lm}, \frac{y}{\lm}, \frac{z}{\lm} \Big) \Big)
= (\mA_{\lm,\mu} W) \big( (\mA_{\lm,\mu} P)(x,y,z) \big)
\quad \ \forall (x,y,z) \in \mA_{\lm,\mu} \mN
\end{align*}
and $(\mA_{\lm,\mu} \tilde P)(x,y,z) = \mu^2 W(2\e) 
= (\mA_{\lm,\mu} W)(2 \mA_{\lm,\mu} \e)$ 
for all $(x,y,z) \in \R^3 \setminus (\mA_{\lm,\mu} \mN)$.

\item
Both $\mA_{\lm,\mu} \tilde U$ and $\mA_{\lm,\mu} \tilde P$ are $C^\infty(\R^3)$ 
and satisfy the Euler equations and the localizability condition in $\R^3$.

\item
Both $\mA_{\lm,\mu} \tilde U$ and $\grad(\mA_{\lm,\mu} \tilde P)$ vanish outside
the bounded set $\mA_{\lm,\mu} \mS$.

\item
One has $\mA_{\lm,\mu} P > \mA_{\lm,\mu} \tau$ in 
$(\mA_{\lm,\mu} \mN) \setminus (\mA_{\lm,\mu} \mN')$.

\item 
One has $\mA_{\lm,\mu} \tau \geq 3 \mA_{\lm,\mu} \e$.
\end{itemize}

Thus, we have obtained the 2-parameter family $\{ \mA_{\lm,\mu} \mU \}_{\lm, \mu}$. 
One has the group property
\begin{equation}  \label{group}
\mA_{\lm_1, \mu_1} (\mA_{\lm_2, \mu_2} \mU) = \mA_{\lm_1 \lm_2 , \, \mu_1 \mu_2} \mU,
\quad \ 
\mA_{\lm, \mu} (\mA_{\frac{1}{\lm}, \frac{1}{\mu}} \mU) 
= \mA_{1,1} \mU = \mU
\end{equation}
for all $\lm_1, \lm_2, \lm, \mu_1, \mu_2, \mu \in (0,\infty)$.
The check of \eqref{group} is straightforward.

\begin{lemma}  \label{lemma:fam}
Given $R>0$, let $\mU_R$ be the list \eqref{def.mU}
of the elements (sets, constants, and functions) 
defined in Section \ref{subsec:Gav.sol}. 
Then 
\[
\mU_R = \mA_{R, R^2} \, \mU_1
\] 
where $\mU_1$ is a list of elements constructed in Section \ref{subsec:Gav.sol} for $R=1$, 
and $\mA_{R, R^2}$ is the rescaling operator $\mA_{\lm,\mu}$, defined in \eqref{def.rescaling.op}, 
at $\lm = R$, $\mu = R^2$.
\end{lemma}

\begin{proof}
Let $\mU_R$ be the list \eqref{def.mU} of elements constructed in Section \ref{subsec:Gav.sol} 
for a given $R>0$. We observe that the list 
$\mA_{\lm, \mu} \mU_R$ with $\lm = 1/R$ and $\mu = 1/R^2$ 
coincides with a list of elements that one obtains 
by choosing $R=1$ in Section \ref{subsec:Gav.sol}, which we call $\mU_1$. 
The check is elementary; for example, regarding the vector field $U$ 
and the pressure $P$, 
by \eqref{def.U.P} and \eqref{Gav.a.b.p} one has 
\begin{alignat*}{2}
\frac{ u_\rho(R \rho, Rz) }{R^2}
& = \frac{\pa_z \a(\rho,z)}{4 \rho}, \qquad &
\frac{ u_\ph(R \rho, Rz) }{R^2}
& = \frac{\sqrt{H(\a(\rho,z))}}{4 \rho}, 
\\
\frac{ u_z(R \rho, Rz) }{R^2}
& = - \frac{\pa_\rho \a(\rho,z)}{4 \rho}, \qquad & 
\frac{P(Rx, Ry, Rz)}{R^4} 
& = \frac{\a(\rho,z)}{4}.
\end{alignat*}
Then, by \eqref{group}, 
$\mU_R = \mA_{R, R^2} ( \mA_{\frac{1}{R}, \frac{1}{R^2}} \, \mU_R ) 
= \mA_{R, R^2} \, \mU_1$.
\end{proof}

Regarding the fluid particle system \eqref{ode.xyz}, 
the consequence of Lemma \ref{lemma:fam} is the following lemma, 
whose proof is trivial.

\begin{lemma} \label{lemma:ode.R}
Let $\tilde U_R = \mA_{R, R^2} \tilde U_1$, 
where $\tilde U_R$ is given by Section \ref{subsec:Gav.sol} for some $R>0$ 
and $\tilde U_1$ is given by Section \ref{subsec:Gav.sol} for $R=1$.  
Then a function $(x(t), y(t), z(t))$ solves 
the fluid particle system \eqref{ode.xyz} with velocity field $\tilde U = \tilde U_R$
if and only if 
\begin{equation*} 
x(t) = R x_1( R t), \quad \ 
y(t) = R y_1( R t), \quad \ 
z(t) = R z_1( R t), 
\end{equation*}
where $(x_1(t), y_1(t), z_1(t))$ solves 
\begin{equation*} 
(\dot x_1(t), \dot y_1(t), \dot z_1(t))
= \tilde U_1(x_1(t), y_1(t), z_1(t)).
\end{equation*}
\end{lemma}

\section{Conjugation to a linear flow on $\T^2$}
\label{sec:conj}

In this section and in Section \ref{sec:Taylor} we prove Theorem \ref{thm:main}. 
Thus, assume that $\mU$ in \eqref{def.mU} is given by Section \ref{subsec:Gav.sol} for $R=1$.
Hence, in particular,
\begin{align}
\mC & = \{ (x,y,z) \in \R^3 : \rho = 1, \ z = 0 \},
\notag \\ 
\mN & = \{ (x,y,z) \in \R^3 : (\rho - 1)^2 + z^2 < \d^2 \},
\notag \\
\mN' & = \{ (x,y,z) \in \R^3 : (\rho - 1)^2 + z^2 < (\d / 4)^2 \},
\label{def.mN.1}
\end{align}
where $\rho = \sqrt{x^2 + y^2}$.
The constant $\d$ satisfies $0 < \d < 1$ and $\d \leq r_0$.
The vector field $U$ and the pressure $P$ in $\mN$ are given by \eqref{def.U.P}, \eqref{Gav.a.b.p}, 
where $a(\rho,z) = \a(\rho,z)$ and 
\begin{equation} \label{mela}
p(\rho,z) = \frac{\a(\rho,z)}{4}, \quad \ 
b(\rho,z) = \frac14 \sqrt{H(\a(\rho,z))}.
\end{equation} 
The constants $\tau$ and $\e$ satisfy \eqref{P.larger} and $\tau \geq 3 \e$.  
The function $\om$ is supported in $[\e, 2 \e]$, and $W$ is in \eqref{def.W}.
The vector field $\tilde U$ and the pressure $\tilde P$ 
are given by \eqref{def.tilde.U.tilde.P} in $\mN$, 
and $(\tilde U, \tilde P) = (0, W(2\e))$ in $\R^3 \setminus \mN$. 
The sets $\mS, \mS^*$ are given by \eqref{def.mS}, \eqref{def.mS.star}.

Outside the set $\mS$, the vector field $\tilde U$ is zero, 
and the solutions of system \eqref{ode.xyz} are constant in time. 
Hence, by the uniqueness property of the solution of Cauchy problems, 
any solution of \eqref{ode.xyz} that is in $\mS$ at some time $t_0$ 
remains in $\mS$ for its entire lifespan; 
in other words, $\mS$ is an invariant set for \eqref{ode.xyz}. 
Moreover $\mS$ is bounded, and therefore, by basic \textsc{ode} theory, 
the solutions of \eqref{ode.xyz} in $\mS$ are all global in time. 

Trivially, any subset of $\R^3 \setminus \mS$ 
is also an invariant set for \eqref{ode.xyz}. 
Hence any subset of $\R^3$ containing $\mS$ is invariant for \eqref{ode.xyz}. 
In particular, $\mS^*$ and $\mN$ are invariant for \eqref{ode.xyz}.

\subsection{Cilindrical coordinates}
\label{subsec:cil.coord}
 
To study system \eqref{ode.xyz} in $\mN$, 
first of all we move on to cilindrical coordinates, that is,
we consider the diffeomorphism 
\begin{align}
\Phi_1 & : \mN_1 \to \mN, \quad \ 
\Phi_1(\rho, \ph, z) := (\rho \cos \ph, \, \rho \sin \ph, \, z), 
\notag \\ 
\mN_1 & := \{ (\rho, \ph, z) \in \R \times \T \times \R : 
(\rho - 1)^2 + z^2 < \d^2 \}, 
\label{def.Phi.1}
\end{align}
and the change of variables 
$\tilde v = \Phi_1(v)$, 
with $\tilde v = (x,y,z)$, 
$v = (\rho, \ph, z)$, 
i.e., $x = \rho \cos \ph$, $y = \rho \sin \ph$. 
Using cilindrical coordinates is a natural choice because 
the quantities $u_\rho, u_\ph, u_z, p$ in \eqref{Gav.a.b.p}
are already expressed in terms of $\rho,z$. 
Now a function $\tilde v(t) = \Phi_1(v(t))$ 
solves \eqref{ode.xyz} in $\mN$ 
if and only if $v(t)$ solves 
\begin{equation} 
\dot v(t) = V_1(v(t))
\label{syst.V.1}
\end{equation}  
in $\mN_1$, where $V_1$ is the vector field 
\begin{equation*} 
V_1(v) := ( D \Phi_1(v) )^{-1} \tilde U(\Phi_1(v)), \quad \ 
v = (\rho, \ph, z) \in \mN_1.
\end{equation*}  
The Jacobian matrix $D \Phi_1(v)$ and its inverse matrix are 
\[
D \Phi_1(v) = 
\begin{bmatrix} 
\cos \ph & - \rho \sin \ph & 0 \\
\sin \ph & \rho \cos \ph & 0 \\
0 & 0 & 1 
\end{bmatrix}, 
\quad \ 
(D \Phi_1(v))^{-1} = 
\begin{bmatrix} 
\cos \ph & \sin \ph & 0 \\
- \frac{1}{\rho} \sin \ph & \frac{1}{\rho} \cos \ph & 0 \\
0 & 0 & 1 
\end{bmatrix}, 
\]
the composition $\tilde U (\Phi_1(v))$ is 
\[
\tilde U( \Phi_1(v)) = \om (p(\rho,z)) U(\Phi_1(v)),
\quad \ 
U(\Phi_1(v))
= \begin{pmatrix} 
u_\rho(\rho,z) \cos \ph - u_\ph(\rho,z) \sin \ph \\
u_\rho(\rho,z) \sin \ph + u_\ph(\rho,z) \cos \ph \\
u_z(\rho,z)
\end{pmatrix},
\]
and therefore 
\begin{align} 
V_1(\rho, \ph, z) 
& = \om(p(\rho,z))
\Big( u_\rho(\rho,z), \ 
\frac{u_\ph(\rho,z)}{\rho}, \ 
u_z(\rho,z) \Big)
\notag \\ 
& = \Big( \frac{\om(p(\rho,z)) \pa_z p(\rho,z)}{\rho}, \ 
\frac{\om(p(\rho,z)) b(\rho,z)}{\rho^2}, \ 
- \frac{\om(p(\rho,z)) \pa_\rho p(\rho,z)}{\rho} \Big).
\label{formula.V.1}
\end{align}  

Since $\a$ and $H$ are the functions constructed and studied in \cite{Gav},
it is convenient to express the other quantities in terms of them. 
Hence, we define 
\begin{equation} \label{def.chi.new}
\chi(s) := \frac14 \om \Big( \frac{s}{4} \Big) \quad \ \forall s \in \R,
\end{equation}
and, recalling \eqref{mela}, we rewrite \eqref{formula.V.1} as 
\begin{equation}  \label{formula.V.1.bis}
V_1(\rho, \ph, z) 
= \chi(\a(\rho,z)) \Big( \frac{ \pa_z \a(\rho,z)}{\rho}, \ 
\frac{ \sqrt{H(\a(\rho,z))}}{\rho^2}, \ 
- \frac{\pa_\rho \a(\rho,z)}{\rho} \Big).
\end{equation}  
The curve $\mC$ and the sets $\mS^*, \mN'$ become 
\begin{align*}
\mC_1 & := \Phi_1^{-1}(\mC) = \{ (1,\ph,0) : \ph \in \T \}, 
\\ 
\mS_1^* & := \Phi_1^{-1}(\mS^*) 
= \{ (\rho,\ph,z) \in \mN_1 : 0 < \a(\rho,z) < 4 \tau \}, 
\\
\mN_1' & := \Phi_1^{-1}(\mN') = \{ (\rho,\ph,z) \in \mN_1 : 
(\rho-1)^2 + z^2 < (\d/4)^2 \}.
\end{align*}
By \eqref{P.larger} and \eqref{mela}, one has 
\begin{equation} \label{p.larger}
\a(\rho,z) > 4 \tau 
\quad \ \forall (\rho,z) \in \mN_1 \setminus \mN_1'.
\end{equation}
Moreover, the map $\Phi_1$ is analytic in $\mN_1$.

\medskip

The vector field $U$ satisfies the localizability condition \eqref{loc}
in $\mN \setminus \mC$, 
and therefore $\tilde U \cdot \grad P = 0$ in $\mN$. 
Hence the pressure $P$ is a prime integral of system \eqref{ode.xyz} 
in $\mN$. 
In cilindrical coordinates, this means that 
$p(\rho,z)$, and therefore $\a(\rho,z)$ too, 
are prime integrals of \eqref{syst.V.1} in $\mN_1$.   
This can also be verified directly: 
by \eqref{syst.V.1} and \eqref{formula.V.1.bis}, 
\begin{align*}
\frac{d}{dt} \big\{ \a( \rho(t), z(t)) \big\} 
& = \pa_\rho \a(\rho, z) \dot \rho + \pa_z \a(\rho, z) \dot z
= 0.
\end{align*}
Hence every trajectory $\{ v(t) = (\rho(t), \ph(t), z(t)) : t \in \R \}$ 
of system \eqref{syst.V.1} in $\mN_1$ 
lies in a level set 
\begin{equation}  \label{def.mP.c}
\mP_c := \{ (\rho,\ph,z) \in \mN_1 : \a(\rho,z) = c \}.
\end{equation}
In particular, every trajectory of \eqref{syst.V.1} in $\mS_1^*$ 
lies in a level set $\mP_c$ with $0 < c < 4 \tau$. 
In fact, the set $\mS_1^*$ is exactly the union of all the level sets 
$\mP_c$ with $c \in (0, 4 \tau)$.

\subsection{Elimination of the factor $1/\rho$ and canonical Hamiltonian structure}

We want to remove the factor $1/\rho$ appearing the first and third component 
of $V_1$ in \eqref{formula.V.1}. 
We consider the diffeomorphism 
\begin{align}
\Phi_2 & : \mN_2 \to \mN_1, \quad \ 
\Phi_2(\rho, \ph, z) = \Big( \rho, \ph, \frac{z}{\rho} \Big), 
\notag \\
\mN_2 & := \{ (\rho,\ph,z) \in \R \times \T \times \R : 
(\rho - 1)^2 + z^2 \rho^{-2} < \delta^2 \}, 
\label{def.Phi.2}
\end{align}
and the change of variables $\tilde v = \Phi_2(v)$, 
with $\tilde v = (\tilde \rho, \tilde \ph, \tilde z)$, 
$v = (\rho,\ph,z)$. 
A function $\tilde v(t) = \Phi_2( v(t))$ solves \eqref{syst.V.1} in $\mN_1$ 
if and only if $v(t)$ solves 
\begin{equation} 
\dot v(t) = V_2(v(t))
\label{syst.V.2}
\end{equation}  
in $\mN_2$, where 
\begin{equation} 
V_2(v) := ( D \Phi_2(v) )^{-1} V_1 (\Phi_2(v)), \quad \ 
v = (\rho, \ph, z) \in \mN_2.
\label{def.V.2}
\end{equation}  
The Jacobian matrix $D \Phi_2(v)$ and its inverse are 
\[
D \Phi_2(v) = 
\begin{bmatrix} 
1 & 0 & 0 \\
0 & 1 & 0 \\
- z \rho^{-2} & 0 & \rho^{-1} 
\end{bmatrix}, 
\quad \ 
(D \Phi_2(v))^{-1} = 
\begin{bmatrix} 
1 & 0 & 0 \\
0 & 1 & 0 \\
z \rho^{-1} & 0 & \rho 
\end{bmatrix}.
\]
We define 
\begin{equation}  \label{def.alpha.2}
\a_2(\rho,z) := \a \Big( \rho, \frac{z}{\rho} \Big) 
\end{equation}
for all $(\rho, z) \in \mD_2$. 
Hence 
\[
\pa_z \a_2(\rho,z) 
= (\pa_z \a) \Big( \rho, \frac{z}{\rho} \Big) \frac{1}{\rho}, 
\quad \ 
\pa_\rho \a_2(\rho,z) 
= (\pa_\rho \a) \Big( \rho, \frac{z}{\rho} \Big) 
- (\pa_z \a) \Big( \rho, \frac{z}{\rho} \Big) \frac{z}{\rho^2}, 
\]
and, recalling the second identity in \eqref{mela}, 
\[
V_1(\Phi_2(\rho,\ph,z)) 
= \chi ( \a_2(\rho, z) )
\Big( \pa_z \a_2(\rho, z), \, 
\frac{\sqrt{H(\a_2(\rho,z))}}{\rho^2}, \, 
- \frac{\pa_\rho \a_2(\rho, z)}{\rho} 
- \frac{z \pa_z \a_2(\rho,z)}{\rho^2}  \Big).
\]
Then the vector field $V_2$ in \eqref{def.V.2} is
\begin{equation*} 
V_2(\rho,\ph,z)
= \chi ( \a_2(\rho, z) )
\Big( \pa_z \a_2(\rho, z), \, 
\frac{\sqrt{H(\a_2(\rho,z))}}{\rho^2}, \, 
- \pa_\rho \a_2(\rho, z) \Big).
\end{equation*}
The sets $\mC_1, \mS_1^*, \mN_1'$ become 
\begin{align*}
\mC_2 & := \Phi_2^{-1}(\mC_1) = \{ (1,\ph,0) : \ph \in \T \} = \mC_1, 
\\ 
\mS_2^* & := \Phi_2^{-1}(\mS_1^*) = \{ (\rho,\ph,z) \in \mN_2 : 
0 < \a_2(\rho,z) < 4 \tau \},
\\
\mN_2' & := \Phi_2^{-1}(\mN_1') = \{ (\rho,\ph,z) \in \mN_2 : 
(\rho - 1)^2 + z^2 \rho^{-2} < (\delta/4)^2 \}.
\end{align*}
Note that $\Phi_2$ leaves $\mC_2 = \mC_1$ invariant, 
because $z/\rho = z$ at $\rho = 1$. 
By \eqref{p.larger}, one has 
\begin{equation}  \label{p.2.larger} 
\a_2(\rho,z) > 4 \tau \quad \ \forall (\rho,\ph,z) \in \mN_2 \setminus \mN_2'.
\end{equation}
The level set $\mP_c$ of $\a$ in \eqref{def.mP.c} becomes the level set 
\begin{equation}  \label{def.mP.2}
\mP_{2,c} = \{ (\rho,\ph,z) \in \mN_2 : \a_2(\rho,z) = c \}
\end{equation}
of $\a_2$. 
The set $\mS_2^*$ is the union of the level sets $\mP_{2,c}$ with $c \in (0, 4 \tau)$. 
The map $\Phi_2$ is analytic in $\mN_2$.
The function $\a_2(\rho,z)$ is a prime integral of system \eqref{syst.V.2}, 
and it is also analytic; 
its Taylor expansion around $(1,0)$ is in \eqref{Taylor.alpha.2.good.red}. 
The first and third equation of system \eqref{syst.V.2} 
are the Hamiltonian system
\begin{equation} \label{Ham.syst.2}
\dot \rho = \pa_z \mH_2(\rho,z), \quad \ 
\dot z = - \pa_{\rho} \mH_2(\rho,z), 
\end{equation}
where 
\begin{equation}  \label{def.mH2.Gamma}
\mH_2(\rho,z) := \Gamma (\a_2(\rho,z)), 
\quad \ 
\Gamma(s) := \int_0^s \chi(\s) \, d \s.
\end{equation}

\subsection{Symplectic polar coordinates in the radial-vertical plane}
\label{subsec:th.xi}

Now we move on to polar coordinates (in their symplectic version) to describe the pair $(\rho,z)$.
The pairs $(\rho,z)$ such that $(\rho, \ph,z) \in \mN_2$ do not form a disc; 
thus, it is convenient to consider a subset of $\mN_2$ 
that fits with polar coordinates better than how $\mN_2$ does. 
We consider the open sets 
\begin{align}  
\mB_2 
& := \{ (\rho,\ph,z) \in \R \times \T \times \R
: 0 < (\rho-1)^2 + z^2 < \d_2^2 \}, 
\notag \\
\mB_2' 
& := \{ (\rho,\ph,z) \in \R \times \T \times \R
: 0 < (\rho-1)^2 + z^2 < (\d_2/2)^2 \}, 
\label{def.mB.2}
\end{align}
where $\d_2 = C \d$, $C>0$, 
and we observe that 
\begin{equation} \label{mB.mN.subset}
(\mB_2 \setminus \mB_2') \subseteq (\mN_2 \setminus \mN_2')
\end{equation}
if $C(1+\d) \leq 1$ and $1 + (\d/4) \leq 2C$. 
This holds, for example, for $C = 2/3$ and all $0 < \d \leq 1/2$.

\begin{proof}[Proof of \eqref{mB.mN.subset}] 
Let $(\rho,\ph,z) \in \mB_2 \setminus \mB_2'$, 
where $\d_2 = C \d$, with 
$C(1+\d) \leq 1$ and $1 + (\d/4) \leq 2C$. 
Since $(\rho,\ph,z) \in \mB_2$, one has 
$(\rho-1)^2 < \d_2^2$, whence $\rho > 1 - \d_2$.  
Moreover $1 - \d_2 = 1 - C \d > 0$ because 
$C \d \leq 1-C < 1$.
Hence
\[
(\rho-1)^2 + \frac{z^2}{\rho^2}
\leq (\rho-1)^2 + \frac{z^2}{(1 - \d_2)^2} 
\leq \frac{(\rho-1)^2 + z^2}{(1 - \d_2)^2} 
< \frac{\d_2^2}{(1 - \d_2)^2}
\leq \d^2,
\]
where the last inequality holds because $C(1+\d) \leq 1$. 
Hence $(\rho,\ph,z) \in \mN_2$. 
Now assume, by contradiction, that $(\rho,\ph,z) \in \mN_2'$. 
Then $(\rho-1)^2 < (\d/4)^2$, whence 
$\rho < 1 + (\d/4)$, 
and
\[
(\rho-1)^2 + z^2 
\leq \Big( (\rho-1)^2 + \frac{z^2}{\rho^2} \Big) \Big( 1 + \frac{\d}{4} \Big)^2 
< \frac{\d^2}{16} \Big( 1 + \frac{\d}{4} \Big)^2 
\leq \frac{\d_2^2}{4},
\]
where the last inequality holds because $1 + (\d/4) \leq 2 C$. 
Also, $0 < (\rho-1)^2 + z^2$ because $(\rho,\ph,z) \in \mB_2$. 
Therefore $(\rho,\ph,z) \in \mB_2'$, a contradiction. 
This proves that $(\rho,\ph,z) \notin \mN_2'$.
\end{proof}

By \eqref{mB.mN.subset} and \eqref{p.2.larger}, one has 
\begin{equation}  \label{p.2.larger.mB}
\a_2(\rho,z) > 4 \tau \quad \ \forall (\rho,\ph,z) \in \mB_2 \setminus \mB_2'.
\end{equation}
Thus $\mS_2^* \subseteq \mB_2' \subseteq \mB_2$, 
and they are invariant sets for system \eqref{syst.V.2}. 

We consider the diffeomorphism 
\begin{align} 
& \Phi_3 : \mB_3 \to \mB_2, \quad \ 
\Phi_3(\th, \ph, \xi) = 
(1 + \sqrt{2\xi} \sin \th, \, \ph, \, \sqrt{2\xi} \cos \th), 
\quad \ 
\mB_3 := \T \times \T \times (0,\xi_3), \quad 
\label{def.Phi.3}
\end{align}
where $\xi_3 := \d_2^2 / 2 = 2 \d^2 / 9$, 
and the change of variables 
$\tilde v = \Phi_3(v)$, 
with $\tilde v = (\tilde \rho, \tilde \ph, \tilde z)$, 
$v = (\th, \ph, \xi)$. 
A function $\tilde v(t) = \Phi_3( v(t))$ solves \eqref{syst.V.2} in $\mB_2$ 
if and only if $v(t)$ solves 
\begin{equation} 
\dot v(t) = V_3(v(t))
\label{syst.V.3}
\end{equation}  
in $\mB_3$, where 
\begin{equation*} 
V_3(v) := ( D \Phi_3(v) )^{-1} V_2 (\Phi_3(v)), \quad \ 
v = (\th, \ph, \xi) \in \mB_3.
\end{equation*}  
The Jacobian matrix $D \Phi_3(v)$ and its inverse are 
\[
D \Phi_3(v) = 
\begin{bmatrix} 
\sqrt{2\xi} \cos \th & 0 & \frac{1}{\sqrt{2\xi}} \sin \th \\
0 & 1 & 0 \\
- \sqrt{2\xi} \sin \th & 0 & \frac{1}{\sqrt{2\xi}} \cos \th
\end{bmatrix}, 
\quad
(D \Phi_3(v))^{-1} = 
\begin{bmatrix} 
\frac{1}{\sqrt{2\xi}} \cos \th & 0 & - \frac{1}{\sqrt{2\xi}} \sin \th \\
0 & 1 & 0 \\
\sqrt{2\xi} \sin \th & 0 & \sqrt{2\xi} \cos \th
\end{bmatrix}.
\]
We define 
\begin{align} 
\a_3(\th,\xi) & := \a_2 \big( 1 + \sqrt{2\xi} \sin \th, \sqrt{2\xi} \cos \th \big), 
\label{def.alpha.3}
\end{align}
for all $(\th,\xi) \in \T \times [0, \xi_3)$. 
In fact, in the set $\mB_3$, $\xi$ varies in the interval $(0, \xi_3)$, 
but it is convenient to consider $\a_3$ also for $\xi=0$; 
one has $\a_3(\th,0) = \a_2(1,0) = \a(1,0) = 0$ (see \eqref{def.alpha.2}).  
From \eqref{def.alpha.3} it follows that
\begin{align}
\pa_\th \a_3(\th,\xi) 
& = (\pa_\rho \a_2)(\rho,z) 
\sqrt{2\xi} \cos \th
- (\pa_z \a_2) (\rho,z)
\sqrt{2\xi} \sin \th,
\notag \\
\pa_\xi \a_3(\th,\xi) 
& = (\pa_\rho \a_2) (\rho,z)
\frac{1}{\sqrt{2\xi}} \sin \th
+ (\pa_z \a_2) (\rho,z)
\frac{1}{\sqrt{2\xi}} \cos \th
\label{der.xi.alpha.3}
\end{align}
for all $(\th,\xi) \in \T \times (0, \xi_3)$, 
where $(\rho,z) = ( 1 + \sqrt{2\xi} \sin \th, \sqrt{2\xi} \cos \th )$.
Hence 
\begin{equation}  \label{formula.V.3}
V_3(\th,\ph,\xi)
= \chi ( \a_3(\th, \xi) )
\Big( \pa_\xi \a_3(\th, \xi), \ 
\frac{\sqrt{H(\a_3(\th,\xi))}}{ (1 + \sqrt{2\xi} \sin \th)^2}, \ 
- \pa_\th \a_3(\th, \xi) \Big).
\end{equation}

The set $\mC_2$ is out of $\mB_2$; 
the sets $\mS_2^*$, $\mB_2'$ become 
\begin{align*}
\mS_3^* & := \Phi_3^{-1}(\mS_2^*) 
= \{ (\th,\ph,\xi) \in \mB_3 : 
0 < \a_3(\th,\xi) < 4 \tau \},
\\
\mB_3' & := \Phi_3^{-1}(\mB_2') 
= \T \times \T \times (0, \xi_3/4).
\end{align*}
By \eqref{p.2.larger.mB}, one has 
\begin{equation}  \label{alpha.3.larger}
\a_3(\th,\xi) > 4 \tau \quad \ \forall 
(\th, \ph, \xi) \in \mB_3 \setminus \mB_3'
= \T \times \T \times [\xi_3/4, \xi_3).
\end{equation}
Thus $\mS_3^* \subset \mB_3'$.
The level set $\mP_{2,c}$ of $\a_2$ in \eqref{def.mP.2} becomes the level set 
\begin{equation*} 
\mP_{3,c} = \{ (\th,\ph,\xi) \in \mB_3 : \a_3(\th,\xi) = c \}
\end{equation*}
of $\a_3$. 
The set $\mS_3^*$ is the union of the level sets $\mP_{3,c}$ with $c \in (0, 4 \tau)$.
The map $\Phi_3$ is analytic in $\mB_3$. 
The function $\a_3(\th,\xi)$ is a prime integral of system \eqref{syst.V.3}; 
its behavior near $\xi=0$ is studied in Sections \ref{subsec:level.sets}
and \ref{sec:Taylor.gamma.c}.
The first and third equation of system \eqref{syst.V.3} are the Hamiltonian system
\begin{equation} \label{Ham.syst.th.xi}
\dot \th = \pa_\xi \mH_3(\th,\xi), \quad \ 
\dot \xi = - \pa_{\th} \mH_3(\th,\xi), 
\end{equation}
where 
\begin{equation}  \label{def.mH.3}
\mH_3(\th,\xi) 
:= \mH_2(1 + \sqrt{2\xi} \sin \th, \sqrt{2\xi} \cos \th )
= \Gamma (\a_3(\th,\xi)),
\end{equation}
where $\mH_2$ and $\Gamma$ are defined in \eqref{def.mH2.Gamma}.
Moreover $\th$ (as well as $\ph$) 
is an angle variable, i.e., 
it varies in $\T = \R / 2 \pi \Z$. 

\begin{remark}
$(i)$ We use the symplectic transformation
$(\rho, z) = (1 + \sqrt{2\xi} \sin \th, \sqrt{2\xi} \cos \th)$,
instead of the simpler polar coordinates $(\rho,z) = (1 + r \cos \th, r \sin \th)$ 
or $(\rho,z) = (1 + r \sin \th, r \cos \th)$,
in order to preserve the canonical Hamiltonian structure of system \eqref{Ham.syst.2}. 

$(ii)$ Using $\cos \th$ for $z$ and $\sin \th$ for $\rho$ 
in the definition \eqref{def.Phi.3} of $\Phi_3$, 
instead of vice versa, is just a matter of convenience: 
in this way, we get a positive angular velocity for the solution $\th(t)$.  
\end{remark}

\subsection{Level sets}
\label{subsec:level.sets}

Since the square root function $\xi \mapsto \sqrt{\xi}$ is analytic in $(0, \infty)$, 
the function $\a_3(\th,\xi)$ defined in \eqref{def.alpha.3} 
is analytic in $(\th,\xi) \in \T \times (0, \xi_3)$.
By the Taylor expansion \eqref{Taylor.alpha.2.good.red} of $\a_2(\rho,z)$, 
by \eqref{def.alpha.3} and \eqref{der.xi.alpha.3},
one has 
\begin{align} 
\a_3(\th,\xi) 
& = 4 \xi + O ( \xi^{\frac32} ), 
\quad \ 
\pa_\xi \a_3(\th,\xi) = 4 + O( \xi^{\frac12} )
\quad \ \text{as } \xi \to 0, \ \xi > 0,
\label{Taylor.alpha.3}
\end{align}
uniformly in $\th \in \T$. 
Note that $\a_3(\th,\xi)$ is not analytic around $\xi=0$, 
namely $\a_3(\th,\xi)$, as a function of $\xi$, 
is not a power series of $\xi$ centered at zero;
in fact, \eqref{Taylor.alpha.3} is deduced from the analyticity of 
$\a_2(\rho,z)$ (and of its partial derivatives $\pa_\rho \a_2(\rho,z)$, 
$\pa_z \a_2(\rho,z)$) around $(1,0)$,  
and then from the evaluation at 
$(\rho,z) = (1 + \sqrt{2\xi} \sin \th, \sqrt{2\xi} \cos \th)$. 
See Section \ref{sec:Taylor.gamma.c} for more details. 

Moreover, $\a_3(\th,0) = \a_2(1,0) = 0$, and, by \eqref{Taylor.alpha.3}, 
\begin{equation} \label{paxi.3}
\pa_\xi \a_3(\th,0) 
= \lim_{\xi \to 0^+} \frac{\a_3(\th,\xi) - \a_3(\th,0)}{\xi} 
= 4 
= \lim_{\xi \to 0^+} \pa_\xi \a_3(\th,\xi).
\end{equation}
Thus $\a_3$ is $C^1$ in $\T \times [0, \xi_3)$. 
Taking $\xi_3$ smaller (i.e., $\d$ smaller) if necessary, 
we have $\pa_\xi \a_3(\th,\xi) > 0$ 
for all $(\th,\xi) \in \T \times [0,\xi_3)$.
Hence the function $\xi \mapsto \a_3(\th,\xi)$ is strictly increasing  
on $[0, \xi_3)$. Moreover $\a_3(\th,0) = 0$ and, 
by \eqref{alpha.3.larger},
$\a_3(\th, \xi_3/4) > 4 \tau$ for all $\th \in \T$. 

As a consequence, 
for every $c \in [0, 4 \tau]$, for every $\th \in \T$,
there exists a unique $\xi \in [0, \xi_3/4)$  
such that $\a_3(\th,\xi) = c$. 
We denote $\g_c(\th)$ 
the unique solution $\xi$ of the equation $\a_3(\th,\xi) = c$. 
Thus, 
\begin{equation}  \label{alpha.3.IFT}
\a_3(\th, \gamma_c(\th)) = c 
\quad \ \forall (\th,c) \in \T \times [0, 4 \tau].
\end{equation}
Moreover, $\g_0(\th) = 0$. 
Since $\a_3$ is analytic in $(\th,\xi) \in \T \times (0, \xi_3)$, 
by the implicit function theorem, the function $\g_c(\th)$ 
is analytic in $(\th,c) \in \T \times (0,4 \tau)$. 
The behavior of $\g_c(\th)$ around $c=0$ is studied in Section \ref{sec:Taylor.gamma.c}.
From \eqref{alpha.3.IFT},
\begin{align} \label{der.th.alpha.3} 
(\pa_\th \a_3)(\th, \gamma_c(\th))
+ (\pa_\xi \a_3)(\th, \gamma_c(\th)) \pa_\th \gamma_c(\th) 
& = 0, 
\\ 
(\pa_\xi \a_3)(\th, \gamma_c(\th)) \pa_c \gamma_c(\th) 
& = 1
\label{der.c.alpha.3} 
\end{align}
for all $(\th,c) \in \T \times (0, 4 \tau)$.
Since $\pa_\xi \a_3(\th,\xi)$ is positive for all $(\th,\xi) \in \T \times [0,\xi_3)$, 
by \eqref{der.c.alpha.3} it follows that 
\begin{equation}  \label{pac.gamma.positivo}
\pa_c \g_c(\th) > 0 
\quad \ \forall (\th,c) \in \T \times (0,4 \tau).
\end{equation} 
 
Thus, for all $c \in [0, 4 \tau]$, 
the level set $\mP_{3,c}$ is globally described as the graph 
\begin{equation}  \label{mP.3.c.graph}
\mP_{3,c} = \{ (\th,\ph,\xi) \in \T \times \T \times [0,\xi_3) 
: \xi = \gamma_c(\th) \} 
= \{ (\th, \ph, \gamma_c(\th)) : (\th,\ph) \in \T^2 \}.
\end{equation}

\subsection{Integration of the system in the radial-vertical plane} 
\label{subsec:solve.th.xi}

We consider the Cauchy problem of system \eqref{syst.V.3} 
with initial datum $(\th_0, \ph_0, \xi_0) \in \mS_3^*$ at the initial time $t=0$. 
The components $\th_0$ and $\xi_0$ of the initial datum 
determine the level $c = \a_3(\th_0,\xi_0)$; 
then the solution $(\th(t), \ph(t), \xi(t))$ of the Cauchy problem
is in the level set $\mP_{3,c}$ for all $t \in \R$, 
and, by \eqref{mP.3.c.graph}, 
$\xi(t) = \g_c(\th(t))$ for all $t \in \R$. 

The first and third equations of system \eqref{syst.V.3} 
are the Hamiltonian system \eqref{Ham.syst.th.xi}, i.e., 
\begin{equation} \label{Ham.syst.chi}
\dot \th = \chi(\a_3(\th,\xi)) \pa_\xi \a_3(\th,\xi), \quad \ 
\dot \xi = - \chi(\a_3(\th,\xi)) \pa_\th \a_3(\th,\xi).
\end{equation}
Since $\xi(t) = \g_c(\th(t))$, 
\eqref{Ham.syst.chi} becomes
\begin{equation} \label{becomes.1.2}
\dot \th = \chi(c) (\pa_\xi \a_3)(\th, \g_c(\th)), 
\quad \ 
\pa_\th \g_c(\th) \dot \th 
= - \chi(c) (\pa_\th \a_3)(\th,\g_c(\th)).
\end{equation}
By \eqref{der.th.alpha.3}, 
the second equation in \eqref{becomes.1.2} 
is equal to the first equation in \eqref{becomes.1.2} 
multiplied by $\pa_\th \gamma_c(\th)$, 
and therefore system \eqref{becomes.1.2} 
is equivalent to its first equation alone. 
Moreover, by \eqref{der.c.alpha.3}, 
the first equation in \eqref{becomes.1.2} is also 
\begin{equation}  \label{becomes.3}
\dot \th = \frac{\chi(c)}{\pa_c \g_c(\th)},
\end{equation}
where the denominator is nonzero by \eqref{der.c.alpha.3} or \eqref{pac.gamma.positivo}.
Equation \eqref{becomes.3} with initial datum $\th(0) = \th_0 \in \T$ 
is an autonomous Cauchy problem for the function $\th(t)$ taking values in $\T$, 
and it can be integrated by basic calculus.

The function $\g_c(\th)$ and its partial derivative $\pa_c \g_c(\th)$ 
are defined for $\th \in \T$, and hence they can also be considered as functions 
defined for $\th \in \R$ that are $2\pi$-periodic in $\th$. 
Thus, we first solve \eqref{becomes.3} considered as an equation 
for a function $\th^r(t)$ taking values in $\R$; 
then the equivalence class of $\th^r(t)$ mod $2\pi$ will be a function $\th(t)$ 
taking values in $\T$ and solving \eqref{becomes.3}. 

For all $c \in (0, 4 \tau)$, we define
\begin{equation}  \label{def.F.c}
F_c : \R \to \R, \quad \ 
F_c(\th) := \int_0^\th \pa_c \gamma_c(\sigma) \, d\sigma.
\end{equation}
By \eqref{pac.gamma.positivo}, $F_c$ is a diffeomorphism of $\R$. 
Let $\th^r_0 \in \R$ be a representative of the equivalence class $\th_0 \in \T$. 
If $\th^r : \R \to \R$, $t \mapsto \th^r(t)$ solves \eqref{becomes.3} 
with $\th^r(0) = \th_0^r$,  
then 
\begin{equation} \label{arancia}
\frac{d}{dt} \{ F_c(\th^r(t)) \}
= F_c'(\th^r(t)) \dot \th^r(t) 
= \chi(c),
\quad \ 
F_c(\th^r(t)) = F_c(\th^r_0) + \chi(c) t,
\end{equation}
and 
\begin{equation}  \label{sol.th}
\th^r(t) = F_c^{-1} \big( F_c(\th^r_0) + \chi(c) t \big).
\end{equation}
Hence $\th^r(t)$ in \eqref{sol.th} is the unique solution 
of equation \eqref{becomes.3} with initial datum $\th^r(0) = \th^r_0$. 
Now let $\th(t)$ be the equivalence class mod $2\pi$ of $\th^r(t)$, 
i.e., 
\begin{equation} \label{th.ph.torus}
\th(t) := \{ \th^r(t) + 2 k \pi : k \in \Z \},  
\end{equation}
for all $t \in \R$. 
Then the function $\th(t)$ in \eqref{th.ph.torus} is a function of time 
taking values in $\T$, 
and it solves \eqref{becomes.3} with $\th(0) = \th_0$.

\subsection{Rotation period in the radial-vertical plane}
\label{subsec:T.c}

We show that the function $\th(t)$ in \eqref{th.ph.torus} is periodic, 
and we calculate its period. 
We begin with observing that the function $F_c$ 
defined in \eqref{def.F.c} satisfies
\begin{equation} \label{Floq.F.c}
F_c(2\pi + \th) = F_c(2\pi) + F_c(\th) 
\quad \ \forall \th \in \R. 
\end{equation}

\begin{proof}[Proof of \eqref{Floq.F.c}] 
Consider the definition \eqref{def.F.c}, 
split the integral over $[0, 2 \pi + \th]$ into the sum 
of $(i)$ the integral over $[0, 2\pi]$ and 
$(ii)$ the integral over $[2\pi, 2 \pi + \th]$;
then $(i) = F(2\pi)$ by definition, 
while $(ii) = F(\th)$ because 
the function $\pa_c \gamma(\s,c)$ is $2\pi$-periodic in $\s$ 
and therefore $(ii)$ is equal to the integral over $[0,\th]$.  
\end{proof}

Suppose that $\chi(c)$ is nonzero.  
Then the function $\th^r(t)$ defined in \eqref{sol.th} satisfies
\begin{equation} \label{Floq.th.r}
\th^r(t + T_c) = \th^r(t) + 2\pi \quad \forall t \in \R, 
\qquad \qquad 
T_c := \frac{F_c(2\pi)}{\chi(c)}.
\end{equation}

\begin{proof}[Proof of \eqref{Floq.th.r}]
Applying \eqref{arancia} twice, one has
\[
F_c(\th^r(t+T_c)) = F_c(\th_0^r) + \chi(c) (t + T_c)
= F_c(\th^r(t)) + \chi(c) T_c.
\]
By the definition of $T_c$ in \eqref{Floq.th.r} and by identity \eqref{Floq.F.c}, 
\[
F_c(\th^r(t)) + \chi(c) T_c
= F_c(\th^r(t)) + F_c(2\pi)
= F_c( \th^r(t) + 2 \pi).
\]
Hence $F_c(\th^r(t + T_c)) = F_c( \th^r(t) + 2\pi)$,
and, since $F_c$ is invertible, we obtain \eqref{Floq.th.r}. 
\end{proof}

From \eqref{Floq.th.r} it follows that 
$\th(t)$ defined in \eqref{th.ph.torus} 
is periodic in time with period $T_c$.

\subsection{Angle-action variables in the radial-vertical plane}

The Hamiltonian system \eqref{Ham.syst.th.xi} 
has one degree of freedom, and therefore, 
by the classical theory of Hamiltonian systems, 
it is completely integrable 
both in the sense that the differential equations can be solved by quadrature 
(i.e., they can be transformed into a problem of finding primitives of functions), 
and in the sense that they admit ``angle-action variables'' 
(this is the one-dimensional, simplest case of the Liouville-Arnold theorem).  
System \eqref{Ham.syst.th.xi} has been integrated in Section \ref{subsec:solve.th.xi}; 
now we construct its angle-action variables. 

\medskip

Given a point $(\th^*,\ph^*,\xi^*) \in \mS_3^*$, 
we calculate its level $c = \a_4(\th^*,\xi^*)$, 
which is in the interval $(0, 12 \e)$, 
and we consider the Cauchy problem 
\begin{equation} \label{Cp.novo}
\dot \th = \frac{1}{\pa_c \g_c(\th)}, \quad \ 
\th(0) = 0.
\end{equation}
The first equation in \eqref{Cp.novo} 
is equation \eqref{becomes.3} in which $\chi(c)$ is replaced by $1$.
Hence, following Section \ref{subsec:solve.th.xi} with 1 instead of $\chi(c)$,
the $\R$-valued solution of \eqref{Cp.novo} 
is given by \eqref{sol.th} in which $\chi(c)$ is replaced by $1$
(note that the functions $\a_3(\th,\xi)$, $\g_c(\th)$, $F_c(\th)$ 
are all independent of $\chi(c)$). 
We indicate $\th^r_c(t)$ the $\R$-valued solution of \eqref{Cp.novo}.  
Thus, since $F_c(0) = 0$, 
\begin{equation} \label{th.r.c}
\th^r_c(t) = F_c^{-1} (t) \quad \ \forall t \in \R.
\end{equation}
Moreover, following Section \ref{subsec:T.c} with 1 instead of $\chi(c)$,
one has 
\begin{equation} \label{T.c.star}
\th^r_c(t + T_c^*) = \th^r_c(t) + 2 \pi 
\quad \forall t \in \R, 
\qquad \qquad 
T^*_c := F_c(2\pi).
\end{equation}

We fix the representative $\th_1^*$ of the class $\th^* \in \T$
such that $\th^*_1 \in [0,2\pi)$, 
and we take the unique real number $s_1^* \in [0,T_c^*)$ 
such that $\th^r_c (s_1^*) = \th^*_1$, 
i.e., by \eqref{th.r.c}, $s_1^* = F_c(\th_1^*)$. 
Then we define $\s_1^* = s_1^* 2 \pi/T_c^*$, 
and we note that $\s_1^* \in [0,2\pi)$ because $s_1^* \in [0,T_c^*)$. 
By the definition of $\s_1^*, s_1^*, T_c^*$, one has 
$\s_1^* = f_c(\th_1^*)$, where 
\begin{equation} \label{def.f.c}
f_c : \R \to \R, \quad \ 
f_c(\th) := \frac{ F_c(\th) 2 \pi }{ F_c(2\pi) } 
\quad \forall \th \in \R.
\end{equation} 
Also, $\th_1^* = \th_c^r(s_1^*)$ and $s_1^* = \s_1^* T_c^* / 2 \pi$, 
whence $\th_1^* = g_c(\s_1^*)$, where 
\begin{equation} \label{def.g.c}
g_c : \R \to \R, \quad \ 
g_c(\s) := \th^r_c \Big( \frac{\s T_c^*}{2\pi} \Big)
\quad \forall \s \in \R.
\end{equation} 

The function $f_c$ is a diffeomorphism of $\R$ 
because it is a multiple of $F_c$, and it satisfies 
\begin{equation} \label{Floq.f.c}
f_c(\th + 2 \pi) = f_c(\th) + 2 \pi  
\quad \ \forall \th \in \R
\end{equation}
because $F_c$ satisfies \eqref{Floq.F.c}.
The function $g_c$ also satisfies 
\begin{equation} \label{Floq.g.c}
g_c(\s + 2 \pi) = g_c(\s) + 2 \pi  
\quad \ \forall \s \in \R
\end{equation}
because $\th^r_c$ satisfies \eqref{T.c.star}.
By \eqref{th.r.c}, \eqref{def.f.c}, \eqref{def.g.c}, 
$g_c(f_c(\th)) = \th$ for all $\th \in \R$
and $f_c(g_c(\s)) = \s$ for all $\s \in \R$,
i.e., $g_c$ is the inverse diffeomorphism of $f_c$.

By \eqref{Floq.f.c} and \eqref{Floq.g.c}, 
$f_c$ and $g_c$ induce diffeomorphisms of the torus
\[
f_c : \T \to \T, \quad \ 
g_c : \T \to \T, \quad \ 
g_c(f_c(\th)) = \th \quad \forall \th \in \T,  \quad \ 
f_c(g_c(\s)) = \s \quad \forall \s \in \T
\]
(with a common little abuse, we use the same notation 
for the diffeomorphisms of $\R$ 
and for the corresponding induced diffeomorphism of $\T$).
Now $\s_1^*$ and $\th_1^*$ are related by the identities 
$\th_1^* = g_c( \s_1^* )$, $\s_1^* = f_c(\th_1^*)$. 
By the definition of $\th_1^*$, 
$\th^* \in \T$ is the equivalence class of $\th_1^*$ mod $2\pi$; 
let $\s^* \in \T$ be the equivalence class of $\s_1^*$ mod $2\pi$. 
Then 
\[
\th^* = g_c(\s^*), \quad \ 
\s^* = f_c(\th^*), \quad \ 
\th^*, \s^* \in \T.
\]
By \eqref{alpha.3.IFT}, $\xi^* = \g_c(\th^*)$. 
Thus, we have expressed $(\th^*, \xi^*)$ in terms of $(\s^*,c)$. 
From now on, we write $\s$ instead of $\s^*$. 

Now we introduce a variable $I$ 
and a function $h$ to express the level $c$ in terms of $I$, 
i.e., $c = h(I)$;
the function $h$ is a diffeomorphism (to be determined) 
of some interval $(0, I^*)$ (to be determined) 
onto the interval $(0, 4 \tau)$ to which $c$ belongs. 
We consider the map 
\begin{equation} \label{def.Phi.4}
\Phi_4(\s,\ph,I) := 
( g_c(\s) , \ph, \g_c(g_c(\s)) )|_{c = h(I)}
= ( g_{h(I)}(\s) , \, \ph, \, \g_{h(I)} (g_{h(I)}(\s)) ),
\end{equation}
defined on the set 
\begin{equation}  \label{def.mS.4}
\mS_4^* := \{ (\s,\ph,I) : \s, \ph \in \T, \, I \in (0, I^*) \} 
= \T \times \T \times (0, I^*),
\end{equation} 
and we want to calculate its Jacobian determinant. 
One has 
\begin{align*}
\pa_\s \{ g_c(\s) \} & = g_c'(\s), \quad \ 
\pa_I \{ g_c(\s) \} = \pa_c g_c(\s) h'(I), \quad \ 
\pa_\s \{ \g_c(g_c(\s)) \} = \g_c'(g_c(\s)) g_c'(\s), 
\\ 
\pa_I \{ \g_c(g_c(\s)) \} 
& = \big\{ (\pa_c \g_c)(g_c(\s)) + \g_c'(g_c(\s)) \pa_c g_c(\s)) \big\} h'(I)
\end{align*}
where $c = h(I)$. Hence the Jacobian matrix is 
\[
D \Phi_4(\s, \ph, I) = \begin{bmatrix} 
g_c'(\s) & 0 & \pa_c g_c(\s) h'(I) \\ 
0 & 1 & 0 \\ 
\g_c'(g_c(\s)) g_c'(\s) & 0 
& \big\{ (\pa_c \g_c)(g_c(\s)) + \g_c'(g_c(\s)) \pa_c g_c(\s)) \big\} h'(I)
\end{bmatrix},
\]
and its determinant is 
\[
\det D \Phi_4(\s, \ph, I) 
= g_c'(\s) \, (\pa_c \g_c)(g_c(\s)) \, h'(I)
\]
where $c = h(I)$. 
By \eqref{def.g.c} and \eqref{Cp.novo}, 
\[
g_c'(\s) = \frac{1}{(\pa_c \g_c)(g_c(\s))} \frac{T_c^*}{2\pi},
\]
and $T_c^* = F_c(2\pi)$, see \eqref{T.c.star}. 
Hence 
\begin{equation}  \label{det.Jac.Phi.4}
\det D \Phi_4(\s, \ph, I) 
= \frac{F_c(2\pi)}{2\pi} h'(I)
\end{equation}
where $c = h(I)$. 
We want the determinant \eqref{det.Jac.Phi.4} to be $= 1$, 
so that the transformation $\Phi_4$ is symplectic 
and the Hamiltonian structure is preserved. 
The average 
\begin{equation} \label{def.h.inv}
h_1(c) := \frac{1}{2\pi} \int_0^{2\pi} \g_c(\th) \, d\th
\end{equation}
is a strictly increasing function of $c \in [0, 4 \tau]$ 
because its derivative $h_1'(c) = F_c(2\pi) / 2\pi$  
is positive for all $c \in (0, 4 \tau)$ 
by \eqref{pac.gamma.positivo} and \eqref{def.F.c}.
Since $\g_0(\th) = 0$, one has $h_1(0) = 0$. 
We define 
\begin{equation} \label{def.I.star}
I^* := h_1(4 \tau),
\end{equation}
and we note that,  by the definition of $\g_c(\th)$, one has $0 < I^* < \xi_3/4$.
Thus, the interval $(0, I^*)$ is the image $\{ h_1(c) : c \in (0 , 4 \tau) \}$ 
of the interval $(0, 4 \tau)$. 
We define $h : (0, I^*) \to (0, 4 \tau)$ as the inverse of $h_1 : (0, 4 \tau) \to (0, I^*)$.  
Thus, 
\begin{equation} \label{der.h.bound}
h'(I) > 0
\quad \ \forall I \in (0, I^*). 
\end{equation}
This concludes the definition of the map $\Phi_4$ in \eqref{def.Phi.4}, 
which is a diffeomorphism of $\mS_4^*$ onto $\mS_3^*$. 

\medskip

Now that the transformation $\Phi_4$ has been defined, 
we use it to transform system \eqref{syst.V.3}. 
We consider the change of variables $\tilde v = \Phi_4(v)$, 
where $\tilde v = (\th, \ph, \xi) \in \mS_3^*$, 
$v = (\s,\ph,I) \in \mS_4^*$. 
This means that 
$\th = g_{h(I)}(\s)$ and 
$\xi = \g_{h(I)}(g_{h(I)}(\s))$, that is, 
$\th = g_c(\s)$ and 
$\xi = \g_c(g_c(\s))$ where $c = h(I)$. 
A function $\tilde v(t) = \Phi_4(v(t))$ solves \eqref{syst.V.3} in $\mS_3^*$ 
if and only if $v(t)$ solves 
\begin{equation} 
\dot v(t) = V_4(v(t))
\label{syst.V.4}
\end{equation}  
in $\mS_4^*$, where 
\begin{equation} 
V_4(v) := ( D \Phi_4(v) )^{-1} V_3 (\Phi_4(v)), \quad \ 
v = (\s, \ph, I) \in \mS_4^*.
\label{def.V.4}
\end{equation}

The second row of the inverse matrix $(D \Phi_4(v))^{-1}$ 
is $(0, 1, 0)$, and therefore the second component of the vector field $V_4$ is simply  
the second component of $V_3$ (see \eqref{formula.V.3}) 
evaluated at $\Phi_4(\s,\ph,I)$, which is 
$\chi(h(I)) Q(\s,I)$, where 
\begin{equation}  \label{def.Q}
Q(\s,I) := \frac{ \sqrt{H(h(I))} }
{ \left( 1 + \sqrt{2 \g_{h(I)} (g_{h(I)}(\s))} \sin(g_{h(I)}(\s)) \right)^2}.
\end{equation}

Since the first and third equations of \eqref{syst.V.4} 
are the Hamiltonian system \eqref{Ham.syst.th.xi} and $\Phi_4$ is symplectic in the $(\s,I)$ variables,  
the first and third components of the vector field $V_4$ are  
$\pa_I \mH_4(\s,I)$ and $- \pa_\s \mH_4(\s,I)$ respectively, where $\mH_4 := \mH_3 \circ \Phi_4$ 
(of course this can also be checked directly, without using the properties of the symplectic transformations). 
Now $\mH_3$ is defined in \eqref{def.mH.3}, 
with $\Gamma$ defined in \eqref{def.mH2.Gamma}. 
Hence 
\begin{align} 
\mH_4(\s,I) 
& = \mH_4(I) = \Gamma(h(I)), 
\quad \ 
\pa_I \mH_4(I) 
= \chi(h(I)) h'(I),
\label{der.mH.4.I}
\end{align}
and system \eqref{syst.V.4} is 
\begin{equation}  \label{syst.V.4.action.angle}
\dot \s = \pa_I \mH_4(I), 
\quad \ 
\dot \ph = \chi(h(I)) Q(\s,I), \quad \ 
\dot I = 0
\end{equation}
in $\mS_4^*$, with $Q(\s,I)$ defined in \eqref{def.Q}.
The action $I$ is constant in time. 
The angle $\s$ rotates with constant angular velocity 
$\pa_I \mH_4(I) = \chi(h(I)) h'(I)$.

\subsection{Reduction to a constant rotation in the tangential direction}

Now that the equations of the motion in the radial-vertical plane 
have been written in angle-action variables $(\s, I)$ in \eqref{syst.V.4.action.angle}, 
we want to obtain a similar simplification for the equation of the motion 
in the tangential direction $\dot \ph$.
Recalling the definition \eqref{def.mS.4} of $\mS_4^*$, 
we consider the diffeomorphism
\begin{equation} \label{def.Phi.5}
\Phi_5 : \mS_4^* \to \mS_4^*, \quad \ 
\Phi_5(\s,\b,I) = (\s, \b + \eta(\s,I), I),
\end{equation}
where $\eta : \T \times (0, I^*) \to \R$ is a function to be determined.
We consider the change of variables $\tilde v = \Phi_5(v)$, 
where $\tilde v = (\s, \ph, I) \in \mS_4^*$, 
$v = (\s,\b,I)$ $\in \mS_4^*$. 
A function $\tilde v(t) = \Phi_5(v(t))$ solves \eqref{syst.V.4} in $\mS_4^*$ 
if and only if $v(t)$ solves 
\begin{equation} 
\dot v(t) = V_5(v(t))
\label{syst.V.5}
\end{equation}  
in $\mS_4^*$, where 
\begin{equation*} 
V_5(v) := ( D \Phi_5(v) )^{-1} V_4 (\Phi_5(v)), \quad \ 
v = (\s, \b, I) \in \mS_4^*.
\end{equation*}  
The Jacobian matrix and its inverse are
\[
D \Phi_5(v) 
= \begin{bmatrix}
1 & 0 & 0 \\
\pa_\s \eta & 1 & \pa_I \eta \\
0 & 0 & 1 
\end{bmatrix}, 
\quad \ 
(D \Phi_5(v))^{-1} 
= \begin{bmatrix}
1 & 0 & 0 \\
- \pa_\s \eta & 1 & - \pa_I \eta \\
0 & 0 & 1 
\end{bmatrix}.
\]
By \eqref{der.mH.4.I} and \eqref{syst.V.4.action.angle}, 
the vector field $V_4$ in \eqref{def.V.4} is 
\[
V_4(\s, \ph, I) = \chi(h(I)) \, ( h'(I), \ Q(\s,I), \ 0).
\] 
Hence 
\[
V_5(\s,\b,I) = \chi(h(I)) \, ( h'(I) , \  Q(\s,I) - \pa_\s \eta(\s,I) h'(I), \ 0 ). 
\]
We decompose $Q(\s,I)$ as the sum of 
its average in $\s \in \T$ 
and its zero-average remainder, that is, 
\begin{equation}  \label{def.Q.0}
Q(\s,I) = Q_0(I) + \tilde Q(\s,I), \quad \ 
Q_0(I) := \frac{1}{2\pi} \int_0^{2\pi} Q(\s,I) \, d\s, \quad \ 
\tilde Q := Q - Q_0,
\end{equation}
and define 
\begin{equation}  \label{def.eta}
\eta(\s,I) := \frac{1}{h'(I)} \int_0^\s \tilde Q(s,I) \, ds.
\end{equation}
Note that $h'(I)$ is nonzero by \eqref{der.h.bound}. 
The function $\eta$ is $2\pi$-periodic in $\s$  
because $\tilde Q$ is $2\pi$-periodic in $\s$ with zero average on $\T$. 
By \eqref{def.eta} and \eqref{def.Q.0},
\[
\pa_\s \eta(\s,I) h'(I) = \tilde Q(\s,I), \quad \ 
Q(\s,I) - \pa_\s \eta(\s,I) h'(I) = Q_0(I). 
\]
Thus, the vector field $V_5$ depends only on the action variable $I$, 
and it is 
\[
V_5(\s, \b, I) = V_5(I) = (\Om_1(I), \Om_2(I), 0), 
\]
where 
\begin{equation}  \label{def.Om.1.2}
\Om_1(I) := \pa_I \mH_4(I) = \chi(h(I)) h'(I),  \quad \ 
\Om_2(I) := \chi(h(I)) Q_0(I).
\end{equation}
Then system \eqref{syst.V.5} is 
\begin{equation}  \label{syst.Om}
\dot \s = \Om_1(I), \quad \ 
\dot \b = \Om_2(I), \quad \ 
\dot I = 0.
\end{equation}

\subsection{Ratio of the two rotation periods}

Let $(\s_0, \b_0, I_0) \in \mS_4^* = \T^2 \times (0, I^*)$. 
Let $\s_0^r \in \R$ be a representative of the equivalence class $\s_0 \in \T$,
and let $\b_0^r \in \R$ be a representative of $\b_0 \in \T$. 
The solution of the Cauchy problem for \eqref{syst.Om} in $\R^2 \times (0, I^*)$
with initial data $(\s_0^r, \b_0^r, I_0)$ 
--- i.e., system \eqref{syst.Om} where the first two equations are considered as equations 
for functions $\s^r(t)$, $\b^r(t)$ taking values in $\R$
with initial data $\s_0^r, \b_0^r$ --- is 
\begin{equation}  \label{sol.good.r}
\s^r(t) = \s_0^r + \Om_1(I_0) t, \quad \ 
\b^r(t) = \b_0^r + \Om_2(I_0) t, \quad \ 
I(t) = I_0 \quad \ 
\forall t \in \R.
\end{equation}
Let $\s(t), \b(t)$ be the equivalence classes of $\s^r(t), \b^r(t)$ mod $2\pi$, i.e., 
\begin{equation}  \label{sol.good.torus}
\s(t) := \{ \s^r(t) + 2 k \pi : k \in \Z \}, \quad \ 
\b(t) := \{ \b^r(t) + 2 k \pi : k \in \Z \}.
\end{equation}
Then $(\s(t), \b(t), I(t))$ is a function of time, 
taking values in $\mS_4^* = \T^2 \times (0, I^*)$, 
solving \eqref{syst.Om} with initial datum $(\s_0, \b_0, I_0)$. 

If $\chi(h(I_0))$ is zero, then both $\Om_1(I_0)$ and $\Om_2(I_0)$ are zero 
by \eqref{def.Om.1.2}, and the solution $(\s(t),\b(t)$, $I(t))$ is constant in time. 
If, instead, $\chi(h(I_0))$ is nonzero, 
then $\Om_1(I_0)$ is also nonzero by \eqref{def.Om.1.2} and \eqref{der.h.bound}, 
and the function $\s(t)$ is periodic in time, 
with frequency $\Om_1(I_0)$ and period $T_1(I_0) := 2 \pi / \Om_1(I_0)$. 
By \eqref{def.Om.1.2}, and recalling that the determinant in \eqref{det.Jac.Phi.4} 
is $=1$, the period $T_1(I_0)$ is equal to the period $T_c$ 
defined in \eqref{Floq.th.r} with $c = h(I_0)$.
If $\Om_2(I_0)$ is nonzero, 
then the function $\b(t)$ is also periodic in time, 
with frequency $\Om_2(I_0)$ and period $T_2(I_0) := 2 \pi / \Om_2(I_0)$.  

We observe that the factor $\chi(h(I))$ cancels out in the frequency ratio 
\begin{equation} \label{ratio}
\frac{\Om_2(I)}{\Om_1(I)}
= \frac{Q_0(I)}{h'(I)},
\end{equation}
and $Q_0(I) / h'(I)$ is well-defined even if $\chi(h(I))$ is zero. 
Hence $Q_0(I) / h'(I)$ is well-defined for all $I \in (0, I^*)$.  
By \eqref{def.Q.0} and \eqref{def.Q},  
\[
\frac{Q_0(I)}{h'(I)} 
= \frac{ \sqrt{H(h(I))} }{h'(I) 2\pi} \int_0^{2\pi} 
\frac{1}{ \left( 1 + \sqrt{2 \g_{h(I)} (g_{h(I)}(\s))} \sin(g_{h(I)}(\s)) \right)^2} \, d\s.
\]
We make the change of variable $g_{h(I)}(\s) = \th$ in the integral, 
that is, $\s = f_c(\th)$, where $c = h(I)$ and $f_c$ is defined in \eqref{def.f.c}. 
Then $d\s = f_c'(\th) d\th$, and, by 
\eqref{def.f.c} and \eqref{def.F.c}, 
$f_c'(\th) = \pa_c \g_c(\th) 2 \pi / F_c(2\pi)$. Hence 
\[
\int_0^{2\pi} 
\frac{1}{ \left( 1 + \sqrt{2 \g_{h(I)} (g_{h(I)}(\s))} \sin(g_{h(I)}(\s)) \right)^2} \, d\s
= \frac{2\pi}{F_c(2\pi)} \int_0^{2\pi}
\frac{\pa_c \g_c(\th)}{[1 + \sqrt{2 \g_c(\th)} \sin \th]^2} \, d\th 
\]
where $c = h(I)$. 
Moreover $h'(I) F_c(2\pi) = 2\pi$  
because $h'(I) F_c(2\pi) / 2\pi$ is the determinant in \eqref{det.Jac.Phi.4},
which is $1$. 
Therefore 
\begin{align} 
\frac{Q_0(I)}{h'(I)} 
& = A(h(I)), 
\qquad \ 
A(c) := \sqrt{H(c)} \, J(c), 
\notag \\ 
J(c) 
& := \frac{ 1 }{ 2\pi } 
\int_0^{2\pi} \frac{\pa_c \g_c(\th)}{[1 + \sqrt{2 \g_c(\th)} \sin \th]^2} \, d\th.
\label{def.A.c}
\end{align}

\subsection{Motion of the fluid particles}
\label{subsec:sol}

We consider the composition 
\begin{equation}  \label{def.Phi}
\Phi := \Phi_1 \circ 
\Phi_2 \circ 
\Phi_3 \circ 
\Phi_4 \circ 
\Phi_5 
\end{equation}
of the transformations defined in 
\eqref{def.Phi.1}, 
\eqref{def.Phi.2}, 
\eqref{def.Phi.3}, 
\eqref{def.Phi.4}, 
\eqref{def.Phi.5}. 
The map $\Phi$ is defined on $\mS_4^* = \T^2 \times (0, I^*)$, 
its image $\Phi(\mS^*_4)$ is the set $\mS^*$ in \eqref{def.mS.star}, 
and $\Phi$ is a diffeomorphism of $\mS_5^*$ onto $\mS^*$. 
The image $(x,y,z) = \Phi(\s,\b,I) \in \mS^*$ 
of a point $(\s,\b,I) \in \mS_4^*$ is 
\begin{equation*} 
x = \rho(\s,I) \cos ( \b + \eta(\s,I) ), 
\quad \ 
y = \rho(\s,I) \sin ( \b + \eta(\s,I) ), 
\quad \ 
z = \zeta(\s,I), 
\end{equation*}
with 
\begin{align}
\rho(\s,I) & := 1 + \sqrt{2 \g_c(g_c(\s))} \sin(g_c(\s)), 
\quad \ 
\zeta(\s,I)  := \frac{ \sqrt{2 \g_c(g_c(\s))} \cos(g_c(\s)) }
{1 + \sqrt{2 \g_c(g_c(\s))} \sin(g_c(\s))}, 
\label{def.zeta}
\end{align}
where $c = h(I)$.  
The analytic regularity of the map $\Phi$ and its expansion around $I=0$ 
are studied in Section \ref{subsec:Taylor.Phi}.

By construction, a function $\tilde v(t) = \Phi(v(t))$ is the solution of 
the Cauchy problem \eqref{ode.xyz} 
with initial datum $\tilde v_0 = \Phi(v_0) \in \mS^*$  
if and only if the function $v(t)$ is the solution of 
\eqref{syst.V.5}, i.e., \eqref{syst.Om}, with initial datum $v_0 \in \mS_4^*$. 
As a consequence, recalling \eqref{sol.good.r}, \eqref{sol.good.torus},
the solution $\tilde v(t)$ 
of the Cauchy problem \eqref{ode.xyz} with initial datum 
$\tilde v_0 = (x_0, y_0, z_0) = \Phi(\s_0, \b_0, I_0)$
is the function 
\begin{equation}  \label{sol.fine.proof}
\tilde v(t) 
= (x(t), y(t), z(t)) 
= \Phi(\s(t), \b(t), I(t))
= \Phi( \s_0 + \Om_1(I_0) t, \b_0 + \Om_2(I_0) t, I_0).
\end{equation}
The function in \eqref{sol.fine.proof} 
has the form $\tilde v(t) = w(\Om_1 t, \Om_2 t)$ 
where $w : \T^2 \to \mS^*$ is the function 
$w(\th_1, \th_2) = \Phi (\s_0 + \th_1, \b_0 + \th_2, I_0)$
and $\Om_i = \Om_i(I_0)$, $i=1,2$. 
Hence $\tilde v(t)$ is quasi-periodic with frequency vector $(\Om_1, \Om_2)$ 
if $\Om_1 \neq 0$, the ratio $\Om_2 / \Om_1$ is irrational,
and the number of frequencies cannot be reduced, i.e., 
if $\tilde v(t)$ is not a periodic function. 

Now suppose that $\Om_1$ is nonzero, 
that $\Om_2 / \Om_1$ is irrational,
and that $\tilde v(t)$ in \eqref{sol.fine.proof} is periodic with a certain period $T > 0$.  
Then, by \eqref{sol.fine.proof},
\[
\begin{pmatrix}
\s_0 + \Om_1 (t + T) \\ 
\b_0 + \Om_2 (t + T) \\ 
I_0 
\end{pmatrix}
= \Phi^{-1}( \tilde v(t+T) ) 
= \Phi^{-1} ( \tilde v(t) ) 
= \begin{pmatrix}
\s_0 + \Om_1 t \\ 
\b_0 + \Om_2 t \\ 
I_0 
\end{pmatrix}
\quad \ \forall t \in \R.
\]
Hence $\s_0 + \Om_1 (t + T)$ and $\s_0 + \Om_1 t$ are the same element of $\T$. 
This means that $\Om_1 T = 2 \pi n$ for some $n \in \Z$. 
Similarly, $\b_0 + \Om_2 (t + T) = \b_0 + \Om_2 t$ in $\T$, 
and $\Om_2 T = 2 \pi m$ for some $m \in \Z$. 
Since $\Om_1$ is nonzero, $n$ is also nonzero, 
and $\Om_2 / \Om_1 = m/n$ is rational, a contradiction. 
This proves that, for $\Om_1$ nonzero and $\Om_2 / \Om_1$ irrational,
the function $\tilde v(t)$ in \eqref{sol.fine.proof} 
is not periodic, and therefore it is quasi-periodic 
with frequency vector $(\Om_1, \Om_2)$.

\subsection{The pressure in terms of the action}

The pressure and the action are related in the following way. 
By \eqref{def.U.P}, \eqref{mela} and \eqref{def.Phi.1}, one has 
$P (\Phi_1 (\rho, \ph, z)) = p(\rho,z) = \frac14 \a(\rho,z)$. 
By \eqref{def.Phi.2}, \eqref{def.alpha.2},
\[
P(\Phi_1(\Phi_2(\rho, \ph, z))) 
= P( \Phi_1(\rho, \ph, z \rho^{-1} )) 
= \frac{1}{4} \a ( \rho, z \rho^{-1} ) 
= \frac{1}{4} \a_2(\rho,z).
\]
By \eqref{def.Phi.3} and \eqref{def.alpha.3},
\[
P( \Phi_1(\Phi_2(\Phi_3(\th, \ph, \xi)))) 
= \frac{1}{4} \a_2 \big( 1 + \sqrt{2\xi} \sin \th, \sqrt{2 \xi} \cos \th \big)
= \frac{1}{4} \a_3(\th,\xi).
\]
By \eqref{alpha.3.IFT} and \eqref{def.Phi.4}, 
\begin{align*}
P( \Phi_1(\Phi_2(\Phi_3( \Phi_4(\s, \ph, I) )))) 
& = \frac{1}{4} \a_3 ( g_c(\s), \g_c(g_c(\s)) )|_{c = h(I)}
= \frac{1}{4} c |_{c = h(I)} 
= \frac{1}{4} h(I).
\end{align*}
By \eqref{def.Phi.5} and \eqref{def.Phi}, 
\begin{equation} \label{P.Phi.h}
P(\Phi(\s,\b,I)) = \frac{1}{4} h(I).
\end{equation}

\section{Taylor expansions and transversality}
\label{sec:Taylor}

In this section we prove that 
the frequency $\Om_1(I)$ 
and the ratio $\Om_2(I) / \Om_1(I)$, 
see \eqref{def.Om.1.2} and \eqref{ratio},
admit a Taylor expansion around $I=0$ 
(which, in principle, is not obvious because of the square roots in the construction), 
and we calculate the first nonzero coefficient in their expansion 
after the constant term. 
This forces us to expand the function $\a$ to degree 6.
As a consequence, we obtain that $\Om_1(I)$ and $\Om_2(I) / \Om_1(I)$
really change as $I$ changes, and they vary in a smooth, strictly monotonic way, 
passing across every value of an interval exactly one time.
This can be geometrically viewed as a transversality property.
We also calculate the expansion of $\Phi(\s, \b, I)$ around $I=0$.

\subsection{Expansion of $\a$}

The function $\a$ is constructed in \cite{Gav} by solving 
the \textsc{pde} system \eqref{PDEs}, 
where the functions $F$ and $G$ 
are expressed in terms of a function $\psi$, 
see \eqref{def.H.F.Gav} and \eqref{def.G.Gav}.
The function $\psi$ is defined in \cite{Gav} as the solution 
of a degenerate \textsc{ode} problem.
In Section 2.1 of \cite{Gav} one finds the Taylor expansion 
of $\psi$ of order 5 around zero; 
here we only use its expansion of order 3, which is 
\begin{equation} \label{Taylor.psi}
\psi(s) = 1 - \frac34 s + \frac{9}{128} s^2 
- \frac{21}{1024} s^3 + O(s^4)
\quad \text{as } s \to 0.
\end{equation} 
In Section 2.1 of \cite{Gav} 
the functions $H, F, G$ are defined in terms of $\psi$ as 
\begin{align}
H(s) & = 6 s \Big( \frac{1}{\psi'(s)} + 2 \psi(s) \Big),
\qquad 
F(x,s) 
= - 2 x \psi(s) + 2 x^3,
\label{def.H.F.Gav}
\\ 
G(x,s) & = 12 x^2 s - F^2(x,s) - H(s).
\label{def.G.Gav}
\end{align}
In Lemma 3 of \cite{Gav} the analytic function $\a(x,y)$ is defined as the unique 
solution of the system 
\begin{equation}  \label{PDEs}
\pa_x \a(x,y) = F(x, \a(x,y)), \quad \ 
\big( \pa_y \a(x,y) \big)^2 = G(x, \a(x,y))
\end{equation}
in a neighborhood of $(x,y) = (1,0)$
such that $\a(1,0) = 0$, with $\pa_y \a$ not identically zero.  
In Remark 2 of \cite{Gav} it is observed that $\a$ is even in $y$, 
i.e., $\a(x,y) = \a(x, -y)$.
The coefficients of the monomials of degree 2 and 3 in the Taylor series
\begin{align}
\a(x,y) 
& = 2 (x-1)^2 
+ 2 y^2 
+ 3 (x-1)^3 
+ 3 (x-1) y^2 
+ \sum_{\begin{subarray}{c} k,j \geq 0 \\ k + 2j \geq 4 \end{subarray}} 
\a_{k, 2j} (x-1)^k y^{2j}
\label{Taylor.alpha.red}
\end{align}
are given in Remark 4 of \cite{Gav};
here we want to calculate the coefficients 
of the monomials of degree $4, 5$ and $6$. 
Monomials with odd exponent $2j+1$ are not present in \eqref{Taylor.alpha.red}
because $\a$ is even in $y$. 
One has
\begin{align}
\pa_x \a(x,y) 
& = 4 (x-1) 
+ 9 (x-1)^2 
+ 3 y^2 
+ 4 \a_{40} (x-1)^3 
+ 2 \a_{22} (x-1) y^2 
\notag \\ 
& \quad \ 
+ 5 \a_{50} (x-1)^4 
+ 3 \a_{32} (x-1)^2 y^2 
+ \a_{14} y^4
+ 6 \a_{60} (x-1)^5
\notag \\ 
& \quad \ 
+ 4 \a_{42} (x-1)^3 y^2
+ 2 \a_{24} (x-1) y^4
+ O_6,
\label{pax.alpha.red}
\end{align}
where $O_n$ denotes terms with homogeneity $\geq n$ in $(x-1, y)$. 
Since $\a(x,y) = O_2$,
by \eqref{Taylor.psi} and \eqref{Taylor.alpha.red} one has 
\begin{align*}
\psi(\a(x,y)) 
& = 1 - \frac34 \a(x,y) 
+ \frac{9}{128} \a^2(x,y) + O_6
\\ & 
= 1 - \frac32 (x-1)^2 
- \frac32 y^2 
- \frac94 (x-1)^3 
- \frac94 (x-1) y^2 
+ \Big( \frac{9}{32} - \frac34 \a_{40} \Big) (x-1)^4 
\\ & \quad \ 
+ \Big( \frac{9}{16} - \frac34 \a_{22} \Big) (x-1)^2 y^2 
+ \Big( \frac{9}{32} - \frac34 \a_{04} \Big) y^4
+ \Big( \frac{27}{32} - \frac34 \a_{50} \Big) (x-1)^5
\\ & \quad \ 
+ \Big( \frac{27}{16} - \frac34 \a_{32} \Big) (x-1)^3 y^2
+ \Big( \frac{27}{32} - \frac34 \a_{14} \Big) (x-1) y^4
+ O_6.
\end{align*}
Expanding $x$ and $x^3$ around $x=1$, 
and recalling the definition \eqref{def.H.F.Gav} of $F$,
we calculate 
\begin{align}
F(x, \a(x,y))
& = - 2 x \psi(\a(x,y)) + 2 x^3
\notag \\ 
& = - 2 \psi(\a(x,y)) 
- 2 (x-1) \psi(\a(x,y))
+ 2 
+ 6 (x-1) 
+ 6 (x-1)^2 
+ 2(x-1)^3
\notag \\ 
& = 
4 (x-1) 
+ 9 (x-1)^2 
+ 3 y^2 
+ \frac{19}{2} (x-1)^3 
+ \frac{15}{2} (x-1) y^2 
\notag \\ & \quad \ 
+ \Big( \frac{63}{16} + \frac32 \a_{40} \Big) (x-1)^4 
+ \Big( \frac{27}{8} + \frac32 \a_{22} \Big) (x-1)^2 y^2 
+ \Big( - \frac{9}{16} + \frac32 \a_{04} \Big) y^4
\notag \\ & \quad \ 
+ \Big( - \frac94 + \frac32 \a_{50} + \frac32 \a_{40} \Big) (x-1)^5
+ \Big( \frac94 + \frac32 \a_{32} + \frac32 \a_{22} \Big) (x-1)^3 y^2 
\notag \\ & \quad \ 
+ \Big( \frac98 + \frac32 \a_{14} + \frac32 \a_{04} \Big) (x-1) y^4
+ O_6.
\label{F.x.alpha.red}
\end{align}
From \eqref{pax.alpha.red}, \eqref{F.x.alpha.red}
and the identity of each monomial in the differential equation $\pa_x \a = F(x,\a)$ 
we get 
\begin{align*}
\a_{40} & = \frac{19}{8}, \quad \ 
\a_{22} = \frac{15}{4}, \quad \ 
\a_{50} = \frac32, \quad \ 
\a_{32} = 3, \quad \ 
\a_{14} = - \frac{9}{16} + \frac34 \a_{04},
\\ 
\a_{60} & = \frac{19}{32}, \quad \ 
\a_{42} = \frac{99}{32}, \quad \ 
\a_{24} = \frac{9}{16} + \frac34 (\a_{14} + \a_{04}).
\end{align*}

To find the value of $\a_{04}$ and $\a_{06}$, 
we consider $\pa_y \a(x,y)$ and $G(x, \a(x,y))$ at $x=1$
and we expand them around $y=0$. 
By \eqref{Taylor.alpha.red}, since $\a$ is even in $y$, 
one has
\begin{equation} \label{Taylor.alpha.y.only.red}
\a(1,y) = 2 y^2 
+ \a_{04} y^4 
+ \a_{06} y^6
+ O(y^8), 
\end{equation}
whence 
\begin{align} 
\pa_y \a(1,y) 
& = 4 y + 4 \a_{04} y^3 + 6 \a_{06} y^5 + O(y^7),
\notag \\
(\pa_y \a(1,y))^2 
& = 16 y^2 + 32 \a_{04} y^4 
+ (48 \a_{06} + 16 \a_{04}^2) y^6 
+ O(y^8).
\label{pay.alpha.square.red}
\end{align} 
By \eqref{F.x.alpha.red},
\begin{align*}
F(1, \a(1,y)) 
& = 3 y^2 + \Big( - \frac{9}{16} + \frac32 \a_{04} \Big) y^4 + O(y^6),
\\
F^2(1, \a(1,y)) 
& = 9 y^4 + \Big( - \frac{27}{8} + 9 \a_{04} \Big) y^6 + O(y^8).
\end{align*}
By \eqref{Taylor.psi} and \eqref{def.H.F.Gav},
\begin{equation} \label{Taylor.H.better}
H(s) = 4 s - \frac{21}{2} s^2 + \frac{39}{32} s^3 + O(s^4)
\end{equation}
(it is to obtain \eqref{Taylor.H.better}
that we use the coefficient of $s^3$ in \eqref{Taylor.psi}).
Therefore, by \eqref{Taylor.alpha.y.only.red},
\[
H(\a(1,y)) = 8 y^2 + ( 4 \a_{04} - 42) y^4 
+ \Big(4 \a_{06} - 42 \a_{04} + \frac{39}{4} \Big) y^6 
+ O(y^8).
\]
Hence 
\begin{align}
G(1, \a(1,y)) 
& = 12 \a(1,y) - F^2(1, \a(1,y)) - H(\a(1,y))
\notag \\ 
& = 16 y^2 + (8 \a_{04} + 33) y^4 
+ \Big( 8 \a_{06} + 33 \a_{04} + \frac{105}{8} \Big) y^6 + O(y^8).
\label{G.1.y}
\end{align}
By \eqref{pay.alpha.square.red}, \eqref{G.1.y} 
and the identity of each monomial in the differential equation 
$(\pa_y \a)^2 = G(x,\a)$ at $x=1$ we get 
$\a_{04} = 11/8$ and 
$\a_{06} = 113/160$.
Hence $\a_{14} = 15/32$, 
$\a_{24} = 249/128$, and 
\begin{align}
\a(x,y) 
& = 2 (x-1)^2 
+ 2 y^2 
+ 3 (x-1)^3 
+ 3 (x-1) y^2 
+ \frac{19}{8} (x-1)^4 
+ \frac{15}{4} (x-1)^2 y^2 
\notag \\ 
& \quad \ 
+ \frac{11}{8} y^4 
+ \frac32 (x-1)^5
+ 3 (x-1)^3 y^2 
+ \frac{15}{32} (x-1) y^4
+ \frac{19}{32} (x-1)^6 
\notag \\
& \quad \  
+ \frac{99}{32} (x-1)^4 y^2 
+ \frac{249}{128} (x-1)^2 y^4 
+ \frac{113}{160} y^6
+ O_7. 
\label{Taylor.alpha.good.red}
\end{align}

\subsection{Expansion of $\a_2$}

The function $\a_2(\rho,z) = \a(\rho, z/\rho)$ 
defined in \eqref{def.alpha.2} 
is even in $z$, it is analytic around $(\rho,z) = (1,0)$, 
and its Taylor series 
\begin{equation} \label{Taylor.series.alpha.2}
\a_2(\rho,z) = \sum_{\begin{subarray}{c} 
k,j \geq 0 \\ k + 2j \geq 2 \end{subarray}} 
(\a_2)_{k,2j} (\rho-1)^k y^{2j}
\end{equation}
matches with the expansion in \eqref{Taylor.alpha.good.red}
in which $x=\rho$ and $y = z / \rho$. 
We expand 
\begin{align*}
\rho^{-2} 
& = 1 - 2 (\rho - 1) + 3 (\rho-1)^2 - 4 (\rho-1)^3 + 5  (\rho-1)^4 + O((\rho-1)^5),
\\
\rho^{-4} 
& = 1 - 4 (\rho-1) + 10 (\rho-1)^2 + O((\rho-1)^3),
\end{align*}
and we obtain  
\begin{align}
\a_2(\rho,z) 
& = 2 (\rho-1)^2 
+ 2 z^2 
+ 3 (\rho-1)^3 
- (\rho-1) z^2 
+ \frac{19}{8} (\rho-1)^4 
+ \frac{15}{4} (\rho-1)^2 y^2 
\notag \\ 
& \quad \ 
+ \frac{11}{8} z^4 
+ \frac32 (\rho-1)^5
- \frac72 (\rho-1)^3 z^2
- \frac{161}{32} (\rho-1) z^4
+ \frac{19}{32} (\rho-1)^6
\notag \\
& \quad \ 
+ \frac{203}{32} (\rho-1)^4 z^2
+ \frac{1529}{128} (\rho-1)^2 z^4
+ \frac{113}{160} z^6
+ O_7. 
\label{Taylor.alpha.2.good.red}
\end{align}

\subsection{Expansion of $\g_c(\th)$}
\label{sec:Taylor.gamma.c}

Recalling the definition \eqref{def.mB.2} of $\mB_2$ 
and the first inclusion in \eqref{mB.mN.subset}, 
the function $\a_2(\rho,z)$ in \eqref{def.alpha.2}
is defined and analytic in the disc 
$(\rho - 1)^2 + z^2$ $< \d_2^2$. 
We define 
\begin{equation}  \label{def.tilde.alpha.3}
\phi : \T \times (- \d_2, \d_2) \to \R^2, \quad \ 
\phi(\th,r) := (1 + r \sin \th, \, r \cos \th), 
\quad \ 
\widetilde \a_3(\th,r) := \a_2(\phi(\th,r)).
\end{equation}
The function $\widetilde \a_3$ is well-posed and analytic 
in $(\th,r) \in \T \times (- \d_2, \d_2)$,  
and, by \eqref{Taylor.series.alpha.2}, it is the power series 
\begin{equation}  \label{tilde.alpha.3.power.series}
\widetilde \a_3(\th,r) = \sum_{n=2}^\infty P_n(\th) r^n, 
\quad \ 
P_n(\th) = \sum_{\begin{subarray}{c} 
k,j \geq 0 \\ k + 2j = n \end{subarray}} 
(\a_2)_{k,2j} (\sin \th)^k (\cos \th)^{2j}.
\end{equation}
All $P_n(\th)$ are trigonometric polynomials, 
and $P_n(-\th) = (-1)^n P_n(\th)$, 
i.e., $P_n$ is even for $n$ even 
and $P_n$ is odd for $n$ odd, because $(-1)^k = (-1)^n$ 
for $k+2j=n$. Hence 
\begin{equation} \label{tilde.alpha.3.even}
\widetilde \a_3(-\th,-r) = \widetilde \a_3(\th, r) 
\quad \ \forall (\th,r) \in \T \times (- \d_2, \d_2),
\end{equation}
namely $\widetilde \a_3$ is an even function of the pair $(\th,r)$. 
In fact, \eqref{tilde.alpha.3.even} is the symmetry property 
$\a_2(\rho,-z) = \a_2(\rho,z)$ 
expressed in terms of the function $\widetilde \a_3(\th,r)$. 
By \eqref{Taylor.alpha.2.good.red}, 
\begin{align}
P_2(\th) & = 2, 
\label{P.2}
\\
P_3(\th) & = 3 \sin^3 \th - \sin \th \cos^2 \th, 
\notag 
\\ 
P_4(\th) & = 
\frac{19}{8} \sin^4 \th 
+ \frac{15}{4} \sin^2 \th \cos^2 \th 
+ \frac{11}{8} \cos^4 \th,
\notag 
\\ 
P_5(\th) & = 
\frac32 \sin^5 \th
- \frac72 \sin^3 \th \cos^2 \th
- \frac{161}{32} \sin \th \cos^4 \th,
\notag 
\\
P_6(\th) & = 
\frac{19}{32} \sin^6 \th
+ \frac{203}{32} \sin^4 \th \cos^2 \th
+ \frac{1529}{128} \sin^2 \th \cos^4 \th
+ \frac{113}{160} \cos^6 \th.
\notag 
\end{align}
Hence 
\begin{align}
P_3(\th) & = 
2 \sin(\th) - \sin(3 \th),
\label{P.3.Fourier}
\\ 
P_4(\th) & = \frac{15}{8} - \frac12 \cos(2 \th),
\label{P.4.Fourier}
\\ 
P_5(\th) & = 
-\frac{33}{256} \sin(\th) - \frac{835}{512} \sin(3 \th) 
+ (P_5)_5 \sin(5\th),
\label{P.5.Fourier}
\\
P_6(\th) & = 
\frac{3173}{2048} 
+ (P_6)_2 \cos(2\th) 
+ (P_6)_4 \cos(4\th) 
+ (P_6)_6 \cos(6\th).
\label{P.6.Fourier}
\end{align}
The Fourier coefficients 
$(P_5)_5$ in \eqref{P.5.Fourier} 
and $(P_6)_k$, $k=2,4,6$, in \eqref{P.6.Fourier} 
are not involved in the calculations we are going to make,
and therefore we avoid to calculate their numerical value.  

When $\widetilde \a_3(\th,r)$ is evaluated at $r = \sqrt{2 \xi}$, 
we obtain the function $\a_3(\th,\xi)$ defined in \eqref{def.alpha.3}, i.e., 
\begin{equation}  \label{alpha.3.tilde.alpha.3}
\a_3(\th,\xi) = \widetilde \a_3(\th, \sqrt{2\xi}).
\end{equation}

The function $\g_c(\th)$ is defined in \eqref{alpha.3.IFT} 
as the unique solution $\xi$ of the equation $\a_3(\th,\xi) = c$.
Because of the square root in the construction, $\g_c(\th)$, as a function of $c$, 
is not analytic around $c=0$, i.e., 
it is not a power series of the form $\sum Q_n(\th) c^n$ 
for some analytic functions $Q_n(\th)$. 
However, $\g_c(\th)$ is a power series of the form $\sum Q_n(\th) c^{n/2}$,
namely there exists a function $\widetilde \g(\th,\mu)$, 
analytic around $\mu = 0$, 
such that $\g_c(\th)$ is $\widetilde \g(\th,\mu)$ 
evaluated at $\mu = \sqrt{c}$. 
To prove it, we use the implicit function theorem for analytic functions, 
taking into account the degeneracy of the problem. 

We define 
\begin{equation} \label{def.mF}
\mF(\th,\mu,w) := \mu^{-2} \widetilde \a_3(\th, \mu w) - 1 
\quad \ \text{if } \mu \neq 0; 
\qquad \ 
\mF(\th,0,w) := 2 w^2 - 1. 
\end{equation}
The function $\mF$ in \eqref{def.mF} is well-defined 
and analytic in $\T \times (- \mu_0, \mu_0) \times (-w_0, w_0)$, 
for some $\mu_0, w_0 > 0$ small enough, 
and, by \eqref{tilde.alpha.3.power.series},
\begin{align*} 
\mF(\th, \mu, w) 
& = \sum_{n=2}^\infty P_n(\th) w^n \mu^{n-2} - 1
= 2 w^2 - 1 + \sum_{n=3}^\infty P_n(\th) w^n \mu^{n-2}.
\end{align*}
Moreover, by \eqref{tilde.alpha.3.even} and \eqref{def.mF},
\begin{equation}  \label{mF.even}
\mF(-\th, -\mu, w) = \mF(\th, \mu, w) 
\quad \ \forall (\th, \mu, w) \in \T \times (- \mu_0, \mu_0) \times (-w_0, w_0).
\end{equation}
For every $\th \in \T$ one has 
\[
\mF ( \th, 0, 2^{-\frac12} ) = 0, 
\quad \ 
\pa_w \mF( \th, 0, 2^{-\frac12}) = 4 \cdot 2^{-\frac12} \neq 0.
\]
Hence there exist two constants $\mu_1, w_1$, 
with $0 < \mu_1 \leq \mu_0$, $0 < w_1 \leq w_0$,
and a function $w(\th,\mu)$, defined and analytic in $\T \times (-\mu_1, \mu_1)$, 
taking values in $(- w_1, w_1)$, such that 
\begin{equation} \label{mF.IFT}
w(\th,0) = 2^{-\frac12} \quad \ \forall \th \in \T,
\qquad 
\mF(\th, \mu, w(\th,\mu)) = 0 
\quad \ \forall (\th, \mu) \in \T \times (- \mu_1, \mu_1),
\end{equation}
and such that if a point $(\th, \mu, a) \in \T \times (- \mu_1, \mu_1) \times (- w_1, w_1)$
is a zero of $\mF$, then $a = w(\th,\mu)$.

By \eqref{mF.IFT} and \eqref{mF.even}, for all $(\th,\mu)$ one has 
\[
0 = \mF(- \th, - \mu, w(-\th, -\mu))
=\mF(\th, \mu, w(-\th, -\mu)). 
\]
Hence the point $(\th, \mu, w(-\th, -\mu)) \in \T \times (- \mu_1, \mu_1) \times (- w_1, w_1)$
is a zero of $\mF$, 
and therefore it belongs to the graph of the implicit function, 
i.e., 
\begin{equation} \label{w.even}
w(-\th, -\mu) = w(\th,\mu) 
\quad \ \forall (\th, \mu) \in \T \times (- \mu_1, \mu_1).
\end{equation}

From the second identity in \eqref{mF.IFT}
and formula \eqref{tilde.alpha.3.power.series}
it follows that $w(\th,r)$ is the power series 
\begin{equation}  \label{w.power.series}
w(\th,\mu) = \sum_{n=0}^\infty W_n(\th) \mu^n, 
\end{equation}
where the functions $W_n(\th)$ are determined by the identity 
$\sum_{n=2}^\infty P_n(\th) w^n(\th,\mu) \mu^{n-2} = 1$,
i.e., $W_0(\th) = 2^{-\frac12}$
and $W_n(\th)$ are trigonometric polynomials 
recursively determined by the system 
\begin{equation} \label{syst.W}
\sum_{\begin{subarray}{c} 
n \geq 2, \ 
j \geq 0, \\
n - 2 + j = m
\end{subarray}}
\ 
\sum_{\begin{subarray}{c} 
k_1, \ldots, k_n \geq 0, \\ 
k_1 + \ldots + k_n = j
\end{subarray}}
P_n(\th) W_{k_1}(\th) W_{k_2}(\th) \cdots W_{k_n}(\th) = 0
\quad \ 
\forall m \geq 1.
\end{equation}
By \eqref{w.even} and \eqref{w.power.series},
 $W_n(-\th) = (-1)^n W_n(\th)$, 
i.e., $W_n$ is even for $n$ even 
and $W_n$ is odd for $n$ odd. 
We will make use of equations \eqref{syst.W} for $m=1,2,3,4$, which are 
\begin{align}
& 2 P_2 W_0 W_1 + P_3 W_0^3 = 0, 
\label{eq.W.1}
\\
& P_2 (2 W_0 W_2 + W_1^2) + 3 P_3 W_0^2 W_1 + P_4 W_0^4 = 0,
\label{eq.W.2}
\\
& P_2 (2 W_0 W_3 + 2 W_1 W_2) + P_3( 3W_0^2 W_2 + 3 W_0 W_1^2) 
+ 4 P_4 W_0^3 W_1 + P_5 W_0^5 = 0,
\label{eq.W.3}
\\
& P_2 (2 W_0 W_4 + 2 W_1 W_3 + W_2^2) 
+ P_3 (3 W_0^2 W_3 + 6 W_0 W_1 W_2 + W_1^3)
\notag \\ & 
+ P_4 (4 W_0^3 W_2 + 6 W_0^2 W_1^2)
+ 5 P_5 W_0^4 W_1
+ P_6 W_0^6 = 0.
\label{eq.W.4}
\end{align} 

By the definition \eqref{def.mF} of $\mF$, 
the second identity in \eqref{mF.IFT} implies that 
\begin{equation} \label{mu.w.IFT}
\widetilde \a_3(\th, \mu w(\th,\mu)) = \mu^2
\quad \ \forall (\th, \mu) \in \T \times (- \mu_1, \mu_1)
\end{equation} 
with $\mu \neq 0$. 
Identity \eqref{mu.w.IFT} also holds for $\mu=0$ 
because $\widetilde \a_3(\th,0) = \a_2(\phi(\th,0)) = \a_2(1,0) = 0$. 

By the first identity in \eqref{mF.IFT},
taking $\mu_1$ smaller if necessary, 
one has $w(\th,\mu) > 0 $ 
for all $(\th, \mu) \in \T \times (- \mu_1, \mu_1)$.
Given any $c \in [0, \mu_1^2)$, there exists a unique $\mu \in [0, \mu_1)$ 
such that $c = \mu^2$, that is, $\mu = \sqrt c$. 
Also, $\mu w(\th,\mu) \geq 0$, 
and there exists a unique $\xi \geq 0$ such that 
$\mu w(\th,\mu) = \sqrt{2 \xi}$, that is, $\xi = \frac12 \mu^2 w^2(\th,\mu)$. 
As a consequence, by \eqref{alpha.3.tilde.alpha.3} and \eqref{mu.w.IFT},
\begin{equation}  \label{see.radice}
\a_3(\th,\xi) 
= \widetilde \a_3(\th, \sqrt{2\xi}) 
= \widetilde \a_3(\th, \mu w(\th,\mu))
= \mu^2
= c.
\end{equation}
Hence $\xi = \g_c(\th)$, and therefore $\g_c(\th) = \frac12 \mu^2 w^2(\th,\mu)$ 
where $\mu = \sqrt c$. 
In other words, we have proved that 
\begin{equation} \label{gamma.tilde.gamma}
\g_c(\th) = \widetilde \g(\th,\sqrt{c}) 
\quad \ \forall (\th, c) \in \T \times [0, \mu_1^2),
\end{equation}
where $\widetilde \g$ is the analytic function 
\begin{equation} \label{def.tilde.gamma}
\widetilde \g(\th, \mu) 
:= \frac12 \mu^2 w^2(\th,\mu).
\end{equation}
By \eqref{w.even}, 
\begin{equation} \label{tilde.gamma.even}
\widetilde \gamma(-\th, -\mu) = \widetilde \gamma(\th,\mu) 
\quad \ \forall (\th, \mu) \in \T \times (- \mu_1, \mu_1).
\end{equation}
By \eqref{def.tilde.gamma} and \eqref{w.power.series}, 
\begin{equation} \label{def.Q.n}
\widetilde \g(\th,\mu) = \sum_{n=2}^\infty Q_n(\th) \mu^n, 
\quad \ 
Q_n(\th) := \frac12 \sum_{\begin{subarray}{c} k,j \geq 0 \\ k + j + 2 = n \end{subarray}} 
W_k(\th) W_j(\th).
\end{equation}
By \eqref{tilde.gamma.even} and \eqref{def.Q.n}, one has 
$Q_n(-\th) = (-1)^n Q_n(\th)$.

\subsection{Expansion of the average of $\g_c(\th)$ and $\pa_c \g_c(\th)$}

We study the average of $\widetilde \g(\th,\mu)$, $\g_c(\th)$ and $\pa_c \g_c(\th)$
over $\th \in [0,2\pi]$. 
To shorten the notation, given any $2\pi$-periodic function $f(\th)$, 
we denote $\la f \ra$ its average over the period, i.e., 
\[
\la f \ra := \frac{1}{2\pi} \int_0^{2\pi} f(\th) \, d\th.
\]
For $n$ odd, the trigonometric polynomial $Q_n(\th)$ in \eqref{def.Q.n} 
is $2\pi$-periodic and odd, and therefore $\la Q_n \ra = 0$. 
As a consequence, by \eqref{def.Q.n}, 
\begin{equation}  \label{int.tilde.gamma}
\frac{1}{2\pi} \int_0^{2\pi} \widetilde \g(\th, \mu) \, d\th
= \sum_{k=1}^\infty \la Q_{2k} \ra \mu^{2k}.
\end{equation}
By \eqref{def.h.inv}, 
\eqref{gamma.tilde.gamma} and \eqref{int.tilde.gamma}, 
one has 
\begin{equation}  \label{int.gamma.c}
h_1(c) = 
\frac{1}{2\pi} \int_0^{2\pi} \g_c(\th) \, d\th
= \frac{1}{2\pi} \int_0^{2\pi} \widetilde \g(\th, \sqrt c) \, d\th
= \sum_{k=1}^\infty 
\la Q_{2k} \ra c^{k}.
\end{equation}
Hence $h_1(c)$ is analytic around $c=0$, 
in the sense that $h_1(c)$, 
which is defined for $c \in [0, \mu_1^2)$, 
coincides in $[0, \mu_1^2)$ 
with the power series in \eqref{int.gamma.c}, 
which is a function defined for $c \in (-\mu_1^2, \mu_1^2)$ 
and analytic in that interval. 
Note that the average of $\g_c(\th)$ is analytic around $c=0$  
even if the function $\g_c(\th)$ itself is not analytic in $c$ around $c=0$. 

Now we calculate the averages $\la Q_n \ra$ for $n=2,4,6$. 
By \eqref{def.Q.n}, 
\begin{equation} \label{Q.2.4.6}
Q_2 = \frac12 W_0^2 = \frac14, \quad \ 
Q_4 = \frac12 (2 W_0 W_2 + W_1^2), \quad \ 
Q_6 = \frac12 (2 W_0 W_4 + 2 W_1 W_3 + W_2^2).
\end{equation}
Hence $\la Q_2 \ra = 1/4$. 
Since $P_2 = 2$ and $W_0 = 2^{-\frac12}$, 
from \eqref{eq.W.1} we get 
\begin{equation} \label{W.1}
W_1 = - \frac18 P_3.
\end{equation}
Using \eqref{eq.W.2} to substitute $(2 W_0 W_2 + W_1^2)$, 
and using also \eqref{W.1}, we calculate
\[
Q_4 
= - \frac14 (3 P_3 W_0^2 W_1 + P_4 W_0^4)
= \frac{1}{64} (3 P_3^2 - 4 P_4).
\]
By \eqref{P.3.Fourier} and \eqref{P.4.Fourier}, we obtain
\begin{equation} \label{int.Q4.nullo}
\la Q_4 \ra = 0. 
\end{equation}
Using \eqref{eq.W.4} to substitute $(2 W_0 W_4 + 2 W_1 W_3 + W_2^2)$,
and also \eqref{W.1}, we calculate
\begin{align}
Q_6 & = - \frac14 \big[ 
P_3 (3 W_0^2 W_3 + 6 W_0 W_1 W_2 + W_1^3)
+ P_4 (4 W_0^3 W_2 + 6 W_0^2 W_1^2)
+ 5 P_5 W_0^4 W_1
+ P_6 W_0^6 \big]
\notag \\
& = - \frac38 P_3 W_3 
+ \frac{3 \sqrt 2}{32} P_3^2 W_2 
+ \frac{1}{2048} P_3^4 
- \frac{\sqrt 2}{4} P_4 W_2
- \frac{3}{256} P_3^2 P_4
+ \frac{5}{128} P_3 P_5
- \frac{1}{32} P_6.
\label{Q6.aux}
\end{align}
From \eqref{eq.W.2} we obtain 
\begin{align} 
W_2
& = \frac{5 \sqrt 2}{128} P_3^2 - \frac{\sqrt 2}{16} P_4
= \frac{\sqrt 2}{256} \Big( -5 - 32 \cos(2 \th) 
+ 20 \cos(4 \th) - 5 \cos(6 \th) \Big)
\label{W.2}
\end{align}
and, from \eqref{eq.W.3} and \eqref{W.1},
\begin{align} 
W_3 & = 
\frac{P_3 P_4 }{16}
- \frac{P_3 W_2 }{2 \sqrt 2} 
- \frac{3 P_3^3}{256} 
- \frac{P_5}{16} 
= \frac{225}{4096} \sin(\th) 
+ \frac{1251}{8192} \sin(3\th) 
+ \sum_{k=5,7,9} (W_3)_k \sin(k \th). \, 
\label{W.3}
\end{align}
The Fourier coefficients $(W_3)_k$, $k=5,7,9$, 
are not involved in the calculation of $\la Q_6 \ra$,
and therefore we do not calculate them. 
By \eqref{P.3.Fourier}, \ldots, \eqref{P.6.Fourier}, \eqref{W.2}, \eqref{W.3},
we calculate
\begin{alignat*}{4}
\la P_3 W_3 \ra & = - \frac{351}{16384}, \qquad &  
\la P_3^2 W_2 \ra & = \frac{291 \sqrt 2}{1024}, \qquad  & 
\la P_3^4 \ra & = \frac{131}{8}, \qquad & 
\la P_4 W_2 \ra & = - \frac{11 \sqrt 2}{2048}, 
\\
\la P_3^2 P_4 \ra & = \frac{91}{16}, \qquad  & 
\la P_3 P_5 \ra & = \frac{703}{1024}, \qquad & 
\la P_6 \ra & = \frac{3173}{2048}.
\end{alignat*}
Hence, integrating \eqref{Q6.aux}, we obtain
\begin{align}
\la Q_6 \ra 
& = - \frac{1065}{65536}
\label{int.Q.6}
\end{align}
(where $65536 = 2^{16}$),
and, by \eqref{int.gamma.c}, 
\begin{equation}  \label{int.gamma.c.finale}
h_1(c) 
= \frac{1}{2\pi} \int_0^{2\pi} \g_c(\th) \, d\th
= \frac{c}{4} + \la Q_6 \ra c^3 + O(c^4).
\end{equation}
Since $\la Q_6 \ra$ in \eqref{int.Q.6} is nonzero,
$h_1(c)$ is a nonlinear, analytic function of $c$.
Its derivative is 
\begin{equation}  \label{int.pac.gamma.c.finale}
h_1'(c) = \frac{1}{2\pi} \int_0^{2\pi} \pa_c \g_c(\th) \, d\th
= \frac14 + 3 \la Q_6 \ra c^2 + O(c^3),
\end{equation}
which is a nonconstant, analytic function of $c$.

\subsection{Expansion of $J(c)$}

The average $J(c)$ is defined in \eqref{def.A.c}. 
By \eqref{gamma.tilde.gamma}, 
the partial derivative $\pa_c \g_c(\th)$ satisfies
\begin{equation} \label{pac.gamma.c.nu}
\pa_c \g_c(\th) = \nu(\th,\sqrt c),
\end{equation}
where $\nu(\th,\mu)$ is the analytic function 
\begin{equation}
\nu(\th,\mu) := \frac{\pa_\mu \widetilde \g(\th,\mu)}{2 \mu} 
= \sum_{n=2}^\infty Q_n(\th) \frac{n}{2} \mu^{n-2},
\label{def.nu}
\end{equation}
with $Q_n(\th)$ defined in \eqref{def.Q.n}.
By \eqref{tilde.gamma.even}, one has 
$\nu(-\th,-\mu) = \nu(\th,\mu)$. Moreover 
\begin{equation} \label{Taylor.nu}
\nu(\th,\mu) = Q_2 + \frac32 Q_3(\th) \mu + 2 Q_4(\th) \mu^2 + O(\mu^3),
\end{equation} 
and $Q_2 = 1/4$ by \eqref{Q.2.4.6}, 
$Q_3 = W_0 W_1$ by \eqref{def.Q.n}, 
$\la Q_3 \ra = 0$ because $Q_3$ is odd,
and $\la Q_4 \ra = 0$ by \eqref{int.Q4.nullo}. 
Regarding the denominator in the definition of $J(c)$, 
one has $\sqrt{2 \g_c(\th)} = \sqrt{2 \xi} 
= \sqrt c \, w(\th, \sqrt c)$ 
by construction (see \eqref{see.radice} and the lines preceding it), 
and, by \eqref{w.power.series}, \eqref{W.1}, \eqref{P.3.Fourier}, 
\begin{align}
\big( 1 + \mu w(\th, \mu) \sin \th \big)^{-2} 
& = 1 - 2 \mu w(\th, \mu) \sin \th + 3 \mu^2 w^2(\th, \mu) \sin^2 \th + O(\mu^3)
\notag \\ 
& = 1 - 2 W_0 \sin (\th) \mu  
+ \big( 3 W_0^2 \sin^2 \th - 2 W_1(\th) \sin \th \big) \mu^2
+ O(\mu^3)
\notag \\ 
& = 1 - \sqrt 2 \sin (\th) \mu  
+ \Big( 1 - \frac98 \cos(2\th) + \frac18 \cos(4\th) \Big) \mu^2
+ O(\mu^3).
\label{Taylor.DEN}
\end{align}
Taking a smaller $\mu_1$ if necessary, 
the function 
\begin{equation} \label{def.m}
m(\th,\mu) := \frac{\nu(\th,\mu)}{[1 + \mu w(\th, \mu) \sin \th]^2}
\end{equation} 
is defined and analytic in $\T \times (- \mu_1, \mu_1)$, 
it satisfies 
\begin{equation} \label{m.even}
m(-\th,-\mu) = m(\th,\mu) 
\quad \ \forall (\th,\mu) \in \T \times (-\mu_1, \mu_1),
\end{equation} 
and it has expansion 
$m(\th,\mu) = \sum_{n=0}^\infty M_n(\th) \mu^n$
for some trigonometric polynomials $M_n(\th)$.
From \eqref{m.even} it follows that $M_n(-\th) = (-1)^n M_n(\th)$. 
Therefore $\la M_n \ra = 0$ for $n$ odd, and
\[
\frac{1}{2\pi} \int_0^{2\pi} m(\th,\mu) \, d\th 
= \sum_{k=0}^\infty \la M_{2k} \ra \mu^{2k}.
\]
By \eqref{Taylor.nu} and \eqref{Taylor.DEN} one has 
\[
M_0 = \frac14, \quad \ 
M_2(\th) = 2 Q_4(\th) - \frac32 W_1(\th) \sin \th 
+ \frac14 - \frac{9}{32} \cos(2\th) + \frac{1}{32} \cos(4\th).
\]
By \eqref{W.1} and \eqref{P.3.Fourier}, the average of 
$W_1(\th) \sin \th$ is $- 1/8$. 
Hence $\la M_2 \ra = 7/16$, and  
\begin{align} 
J(c) 
& = \frac{1}{2\pi} \int_0^{2\pi} 
\frac{\pa_c \g_c(\th)}{[1 + \sqrt{2 \g_c(\th)} \sin \th]^2} \, d\th 
\notag \\ & 
= \frac{1}{2\pi} \int_0^{2\pi} m(\th, \sqrt c) \, d\th 
= \sum_{k=0}^\infty \la M_{2k} \ra c^k 
= \frac14 + \frac{7}{16} c + O(c^2).
\label{Taylor.J}
\end{align}
Thus, the average $J(c)$ is an analytic function of $c$ around $c=0$ 
(even if the integrand function $m(\th, \sqrt c)$ is not). 
Moreover, taking $\mu_1$ smaller if necessary, 
$J(c)$ is strictly increasing in $[0, \mu_1^2)$. 

\subsection{Expansion of $h(I)$ and $\Om_1(I)$}

We have already proved that $h_1(c)$ in \eqref{def.h.inv} 
is analytic around $c=0$, with expansion \eqref{int.gamma.c.finale}. 
Hence its inverse function $h(I) = h_1^{-1}(I)$ 
is also analytic around $I=0$, and it satisfies 
\begin{equation}  \label{Taylor.h}
h(I) = 4 I - 256 \la Q_6 \ra I^3 + O(I^4), 
\quad \ 
h'(I) = 4 - 768 \la Q_6 \ra I^2 + O(I^3).
\end{equation}

The frequency $\Om_1(I)$ is defined in \eqref{def.Om.1.2},  
and it is the product of the $C^\infty$ cut-off function $\chi(h(I))$ 
times the analytic function $h'(I)$ in \eqref{Taylor.h}.

\subsection{Expansion of the frequency ratio $\Om_2(I) / \Om_1(I)$}

By \eqref{ratio} and \eqref{def.A.c}, 
the frequency ratio $\Om_2(I) / \Om_1(I)$ coincides with the function 
\[
A(h(I)) = \sqrt{H(h(I))} \, J(h(I)).
\]
By its definition in \cite{Gav}, the function $H(c)$ is analytic around $c=0$. 
Hence the composition $H(h(I))$ is analytic around $I=0$, 
and, by \eqref{Taylor.H.better} and \eqref{Taylor.h},
\[
H(h(I)) = 16 I - 168 I^2 + O(I^3).
\]
We write its square root as the product 
\begin{equation} \label{H.h.B}
\sqrt{H(h(I))} = \sqrt{I} \, B(I),
\end{equation}
where 
\begin{equation} \label{def.B}
B(I) := 4 \Big( \frac{H(h(I))}{16 I} \Big)^{\frac12}
= 4 \Big( 1 - \frac{21}{2} I + O(I^2) \Big)^{\frac12}
= 4 - 21 I + O(I^2).
\end{equation}
Since the function $x \mapsto \sqrt{1 + x}$ is analytic around $x=0$, 
the function $B(I)$ is analytic around $I=0$.
The function $J(h(I))$ is also analytic around $I=0$, 
and, by \eqref{Taylor.J} and \eqref{Taylor.h}, 
\begin{equation} \label{J.h}
J(h(I)) = \frac14 + \frac74 I + O(I^2).
\end{equation}
Hence 
\[
A(h(I)) = \sqrt{I} \, \mR(I), 
\quad \ 
\mR(I) := B(I) J(h(I))
= 1 + \frac74 I + O(I^2),
\]
and the function $\mR(I)$ is analytic around $I=0$.
Taking a smaller $I^*$ if necessary, 
both $\mR(I)$ and $A(h(I))$ are strictly increasing 
functions of $I \in [0, I^*)$.

\subsection{Expansion of $\Phi(\s, \b, I)$}
\label{subsec:Taylor.Phi}

Since $Q_2 = 1/4$, 
by \eqref{gamma.tilde.gamma}, \eqref{def.Q.n},
\eqref{pac.gamma.c.nu}, \eqref{def.nu}
one has 
\begin{equation*} 
\g_c(\th) = \frac{c}{4} + O(c^{\frac32}), \quad \ 
\pa_c \g_c(\th) = \frac14 + O(c^{\frac12}).
\end{equation*}
Hence, by \eqref{def.F.c}, 
$F_c(\th) = \frac{\th}{4} + O(c^{\frac12})$ 
and therefore, by \eqref{def.f.c}, 
$f_c(\th) = \th + O(c^{\frac12})$. 
As a consequence, $g_c$ in \eqref{def.g.c}, 
being the inverse of $f_c$, satisfies 
\begin{equation*} 
g_c(\th) = \th + O(c^{\frac12}). 
\end{equation*}
By \eqref{Taylor.h}, $h(I) = 4 I + O(I^3)$ and $h'(I) = 4 + O(I^2)$.
Hence, at $c = h(I)$, one has 
\begin{equation*}
\g_c(\th) = I + O(I^{\frac32}), \qquad \ 
g_c(\th) = \th + O(I^{\frac12}), 
\end{equation*}
and therefore the functions $\rho(\s,I), \zeta(\s,I)$ 
in \eqref{def.zeta} have expansion 
\begin{equation}
\rho(\s,I) = \sqrt{2I} \, \sin(\s) + O(I), 
\qquad \ 
\zeta(\s,I) = \sqrt{2I} \, \cos(\s) + O(I).
\label{TM.07}
\end{equation}
By \eqref{H.h.B} and \eqref{def.B}, 
\begin{equation} \label{TM.08}
\sqrt{H(h(I))} = 4 \sqrt{I} + O(I^{\frac32}).
\end{equation}
By \eqref{TM.07} and \eqref{TM.08}, the function $Q(\s,I)$ in \eqref{def.Q} 
satisfies 
\begin{equation*} 
Q(\s,I) = 4 I^{\frac12} + O(I). 
\end{equation*}
Hence $Q_0, \widetilde Q, \eta$ in \eqref{def.Q.0}, \eqref{def.eta} 
satisfy
\begin{equation*} 
Q_0(I) = 4 I^{\frac12} + O(I), \quad \ 
\widetilde Q(\s,I) = O(I), \quad \ 
\eta(\s,I) = O(I).
\end{equation*}

The map $\Phi_1$ in \eqref{def.Phi.1} 
is analytic in $\mN_1$ because $\rho$ does not vanish in $\mN_1$;
the map $\Phi_2$ in \eqref{def.Phi.2} 
is also analytic in $\mN_2$ because $\rho$ does not vanish in $\mN_2$.

The map $\Phi_3$ in \eqref{def.Phi.3} 
is analytic in $\mB_3$ because $\xi$ is positive in $\mB_3$; 
however, $\Phi_3$ is not analytic in $\xi$ around $\xi=0$. 
Nonetheless, $\Phi_3$ can be obtained by evaluating at $r = \sqrt{2 \xi}$ 
a map that is analytic around $r=0$, 
exactly like $\a_3(\th,\xi)$ in \eqref{alpha.3.tilde.alpha.3}, 
like $\g_c(\th)$ in \eqref{gamma.tilde.gamma}, 
and like $\pa_c \g_c(\th)$ in \eqref{pac.gamma.c.nu}.

Similarly, both the map $\Phi_4$ in \eqref{def.Phi.4}
and the map $\Phi_5$ in \eqref{def.Phi.5} 
are analytic in $\mS_4^*$, 
they are not analytic in $I$ around $I=0$, 
but they can be obtained by evaluating at $\mu = \sqrt{I}$ 
some suitable maps that are analytic around $\mu = 0$, 
because $\g_c, F_c, f_c, g_c, \sqrt{H(c)}$ 
are all functions of this type. 

As a consequence, the map $\Phi$ in \eqref{def.Phi} 
is analytic in $\mS_4^*$, 
it is not analytic in $I$ around $I=0$, 
and it can be obtained by evaluating at $\mu = \sqrt{I}$ 
a map that is analytic around $\mu = 0$. 
Hence $\Phi$ admits a converging expansion in powers of $\sqrt{I}$ around $I=0$.

\subsection{Smallness conditions} 

The parameter $\d$ in the definition \eqref{def.mN.1} of the set $\mN$ 
is subject to the following smallness conditions.  
After \eqref{def.mN.1} we have taken $\d \in (0,1)$  
to obtain that $\mN$ is an open neighborhood of the circle $\mC$ in \eqref{def.mN.1} 
with $\rho > 1-\d > 0$ in $\mN$,  
and $\d \leq r_0$ to obtain that the functions $\a(\sqrt{x^2 + y^2},z)$ 
and $H(\a(\sqrt{x^2 + y^2},z))$  
are analytic in $(x,y,z) \in \mN$, where $r_0$ is a universal constant given 
by the definition of $\a$ and $H$, i.e., by Gavrilov's construction in \cite{Gav}.
After \eqref{def.mB.2} we have defined $\d_2 = 2 \d / 3$, 
and we have assumed $\d \leq 1/2$ to obtain the inclusion \eqref{mB.mN.subset}. 
After \eqref{def.Phi.3} we have defined $\xi_3 = \d_2^2/2 = 2 \d^2 / 9$. 
In \eqref{paxi.3} we have proved that $\pa_\xi \a_3(\th,0) = 4$ for all $\th \in \T$, 
and that $\pa_\xi \a_3(\th,\xi)$ is continuous in $\T \times [0, \xi_3)$  
--- in fact, in Section \ref{sec:Taylor.gamma.c} we have proved more, 
because $\a_3(\th,\xi) = \widetilde \a_3(\th, \sqrt{2\xi})$, 
see \eqref{alpha.3.tilde.alpha.3}, 
and $\widetilde \a_3(\th,r)$ is the analytic function 
in \eqref{def.tilde.alpha.3}, \eqref{tilde.alpha.3.power.series}; 
hence $\pa_\xi \a_3(\th,\xi) 
= \sum_{n=2}^\infty P_n(\th) n (2\xi)^{(n-2)/2}
= 4 + 3 P_3(\th) \sqrt{2\xi} + O(\xi)$. 
Hence, by continuity, there exists a universal constant $\xi^* > 0$ such that 
$\pa_\xi \a_3(\th,\xi) > 0$ in $\T \times [0, \xi^*)$. 
Therefore the condition on $\xi_3$ after \eqref{paxi.3} 
(where we say ``Taking $\xi_3$ smaller if necessary'') is $\xi_3 \leq \xi^*$. 
In terms of $\d$, this means $\d \leq 3 \sqrt{\xi^*/2}$, which is a universal constant.

The constants $w_0, \mu_0$ after the definition \eqref{def.mF} of $\mF$ are universal, 
and the constants $w_1, \mu_1$ after the application of the implicit function theorem 
in \eqref{mF.IFT} are universal too. 
By \eqref{mF.IFT}, $w(\th,0) > 0$, and therefore, by continuity, 
there exists a universal constant $\mu_1^* > 0$ such that 
$w(\th,\mu) > 0$ for all $(\th,\mu) \in \T \times (-\mu_1^*, \mu_1^*)$.
Hence the condition on $\mu_1$ after \eqref{mu.w.IFT} 
(where we say ``taking $\mu_1$ smaller if necessary'') 
is $\mu_1 \leq \mu_1^*$. 
The same happens for the condition on $\mu_1$ after \eqref{Taylor.DEN}, 
to obtain that the function $m(\th,\mu)$ in \eqref{def.m} 
is well-defined and analytic in $\T \times (-\mu_1, \mu_1)$, 
and for the condition on $\mu_1$ after \eqref{Taylor.J}, 
to obtain that $J(c)$ is strictly increasing in $[0, \mu_1^2)$.  
Thus, we fix $\mu_1$ as the smallest of these three constants, 
and we obtain that $\mu_1$ is a universal positive constant. 

The parameter $\tau$ is related to $\mu_1$ by the inequality $4 \tau \leq \mu_1^2$, 
because $[0, 4 \tau)$ is the interval where $c$ varies, 
and the functions $\g_c(\th)$, $\pa_c \g_c(\th)$ and $J(c)$ 
are obtained by evaluating at $\mu = \sqrt{c}$ 
some functions of $(\th,\mu)$ that are well-defined, analytic and monotonic
in $\T \times (-\mu_1, \mu_1)$, 
see \eqref{gamma.tilde.gamma}, \eqref{pac.gamma.c.nu}, \eqref{Taylor.J}.
Thus, we want that, for all $c \in [0, 4\tau)$, the square root $\sqrt{c}$ belongs 
to the domain $(-\mu_1, \mu_1)$ of those analytic functions, 
and this is true if $4 \tau \leq \mu_1^2$. 

Regarding the parameter $\tau$, there are two other conditions to consider. 
The first one is \eqref{P.larger}, which holds if $\tau$ is smaller than 
the infimum of the pressure $P$ on the set $\mN \setminus \mN'$. 
Since $P(x,y,z) = p(\rho,z) = \frac14 \a(\rho,z)$ (see \eqref{mela}), 
by the definition \eqref{def.mN.1} of $\mN$ and $\mN'$  
it follows that that infimum depends only on $\d$. 
The last condition for $\tau$ is after \eqref{J.h}, 
where we say ``Taking a smaller $I^*$ if necessary'',   
to obtain that $\mR(I)$ is strictly increasing in $[0, I^*)$. 
Since the function $\mR$ does not depend on any parameter, 
this is a condition of the form $I^* \leq I_0$ for some universal constant $I_0 > 0$. 
Moreover, the invertible function $h_1$ expressing $I^*$ in terms of $\tau$ in \eqref{def.I.star} 
is also independent on any parameter 
(see the definition of $h_1$ in \eqref{def.h.inv} and its expansion in \eqref{int.gamma.c.finale}),
and therefore this condition for $\tau$ is satisfied for $\tau$ smaller 
than a universal constant. 

Regarding $\e$, the only condition to consider is that $0 < \e \leq \tau / 3$, 
see after \eqref{P.larger}.

In conclusion, the parameters $\d, \tau$ and $\e$ must satisfy 
\[
0 < \d \leq \d_0, \quad \  
0 < \tau \leq \tau_0(\d), \quad \ 
0 < \e \leq \tau/3,
\]
where $\d_0$ is a universal constant, 
and $\tau_0(\d)$ depends only on $\d$. 
We fix $\d = \d_0$ and $\tau = \tau_0(\d_0)$. 
Both $\d_0$ and $\tau_0(\d_0)$ are universal constants. 
We rename $\tau_0 := \tau_0(\d_0)$ and $\e_0 := \tau_0/3$. 
Since $\tau = \tau_0$, by \eqref{def.I.star}, $I^*$ is also a universal constant.

For notation convenience, we denote 
\begin{equation*} 
\mK(I) := \frac14 h(I).
\end{equation*}
Hence, by \eqref{def.Om.1.2} and \eqref{def.chi.new}, 
$\Om_1(I) = \frac14 \om ( \frac14 h(I) ) h'(I) 
= \om (\mK(I)) \mK'(I)$;
by \eqref{Taylor.h} and \eqref{int.Q.6}, 
\[
\mK(I) = \frac14 \Big( 4 I + 256 \frac{1065}{2^{16}} I^3 + O(I^4) \Big) 
= I + \frac{1065}{1024} I^3 + O(I^4);
\]
\eqref{P.Phi.h} becomes $P(\Phi(\s,\b,I)) = \mK(I)$.
For all $\e \in (0, \e_0]$, the proof of Theorem \ref{thm:main} is complete.

\bigskip

\bigskip

\textbf{Acknowledgements.} 
Supported by the Project 
PRIN 2020XB3EFL \emph{Hamiltonian and dispersive PDEs}.

\begin{footnotesize}

\end{footnotesize}

\bigskip

Pietro Baldi

\smallskip

Dipartimento di Matematica e Applicazioni ``R. Caccioppoli''

Universit\`a di Napoli Federico II  

Via Cintia, 80126 Napoli, Italy

\smallskip

\texttt{pietro.baldi@unina.it}

\end{document}